\documentclass[11pt,reqno]{amsart} 
\usepackage[a4paper, left=3.4cm, right=3.4cm]{geometry}
\usepackage[a4paper]{geometry}

\usepackage[utf8]{inputenc}
\usepackage[T1]{fontenc}
\usepackage{lmodern}
\usepackage{amsmath,amssymb}
\usepackage{amsthm}
\usepackage{paralist}
\usepackage{float}
\usepackage{accents}
\usepackage{color}
\usepackage[authoryear]{natbib}
\usepackage{bm}
\usepackage{lscape}
\usepackage[position=bottom]{subfig}
\usepackage{graphicx}
\usepackage{booktabs}
\usepackage{tabularx}

\newcolumntype{R}[1]{>{\raggedleft\arraybackslash}p{#1}}
\newcolumntype{L}[1]{>{\raggedright\arraybackslash}p{#1}}

\numberwithin{equation}{section}
\usepackage{dsfont}
\usepackage[colorlinks=true,linkcolor=blue,citecolor=blue,pdfborder={0 0 0}]{hyperref}

\makeatletter
\renewcommand\subsection{\@startsection{subsection}{2}%
	\z@{.5\linespacing\@plus.7\linespacing}{.1em}%
	{\normalfont\bfseries}}
\makeatother

\newcommand{\E}{\mathbb{E}}
\newcommand{\id}{\mathds{1}}

\renewcommand{\epsilon}{\varepsilon}
\newcommand{\eps}{{\varepsilon}}
\renewcommand{\phi}{\varphi}

\newcommand{\R}{\mathbb{R}}
\newcommand{\Z}{\mathbb{Z}}
\newcommand{\N}{\mathbb{N}}

\newcommand{\pr}{\mathbb{P}}        
\newcommand{\Prob}{\mathbb{P}}        
\newcommand{\ex}{\mathbb{E}}        

\newcommand{\Exp}{\mathbb{E}}        
\newcommand{\var}{\textnormal{Var}} 
\newcommand{\cov}{\textnormal{Cov}} 

\newcommand{\MSE}{\textnormal{MSE}}

\newcommand{\Ac}{\mathcal{A}}
\newcommand{\Bc}{\mathcal{B}}
\newcommand{\Dc}{\mathcal{D}}
\newcommand{\Ec}{\mathcal{E}}
\newcommand{\Fc}{\mathcal{F}}

\newcommand{\Oc}{\mathcal{O}}
\newcommand{\Rc}{\mathcal{R}}

\newcommand{\Tc}{\mathcal{T}}

\newcommand{\Gum}{\textnormal{Gum}}

\newcommand{\argmin}{\textnormal{argmin}}

\newcommand{\diff}{{\,\mathrm{d}}}
\newcommand{\supp}{\textnormal{supp}}

\newcommand{\scs}{\scriptscriptstyle}             

\newcommand{\convd}{\rightsquigarrow}              
\newcommand{\convw}{\convd}                           
\newcommand{\weak}{\convd}              
\newcommand{\convp}{\stackrel{\Prob}{\longrightarrow}}              

\newtheorem{theorem}{Theorem}[section]
\newtheorem{proposition}[theorem]{Proposition}

\newtheorem{corollary}[theorem]{Corollary}

\theoremstyle{definition}

\newtheorem{assumption}[theorem]{Assumption}
\newtheorem{alg}[theorem]{Algorithm}
\newtheorem{remark}[theorem]{Remark}

\allowdisplaybreaks[4]

\begin{document}

\title[Relevant changes in a gradually varying mean]{Are deviations in a gradually varying mean relevant? A testing approach based on sup-norm estimators}

\author{Axel B\"ucher}
\author{Holger Dette}
\author{Florian Heinrichs}

\date{\today}

\address{Heinrich-Heine-Universit\"at D\"usseldorf, Mathematisches Institut, Universit\"atsstr.~1, 40225 D\"usseldorf, Germany.}
\email{axel.buecher@hhu.de}

\address{Ruhr-Universit\"at Bochum, Fakult\"at f\"ur Mathematik, Universit\"atsstr.\ 150, 44780 Bochum, Germany.}
\email{holger.dette@rub.de}
\email{florian.heinrichs@rub.de}

\begin{abstract}
Classical change point analysis aims at (1) detecting abrupt changes in the mean of a possibly non-stationary time series and at (2) identifying regions where the mean exhibits a piecewise constant behavior. In many applications however, it is more reasonable to assume that the mean changes gradually in a smooth way. Those gradual changes may either be non-relevant (i.e., small), or relevant for a specific problem at hand, and the present paper presents statistical methodology to detect the latter. More precisely, we consider the common nonparametric regression model  $X_{i}  = \mu (i/n) + \varepsilon_{i}$  with possibly non-stationary errors and propose a test for the null hypothesis that the maximum absolute deviation  of  the regression function $\mu$  from a functional $g (\mu )$  (such as the value $\mu (0)$ or the integral $\int_{0}^{1} \mu (t) dt$) is smaller than a given threshold on a given interval $[x_{0},x_{1}] \subseteq [0,1]$. A test for this type of hypotheses is developed  using an appropriate estimator, say $\hat d_{\infty, n}$, for  the maximum deviation $ d_{\infty  }= \sup_{t \in [x_{0},x_{1}]} |\mu (t)  -g( \mu) |$. We derive the limiting distribution of an appropriately standardized version of $\hat d_{\infty,n}$, where the standardization depends on the Lebesgue measure of the set of extremal points of the function $\mu(\cdot)-g(\mu)$. A refined  procedure based on an estimate of this set is developed and its consistency is proved. The results are illustrated by means of a simulation study and a  data example. 
\end{abstract}

\maketitle

\vspace{-.5cm}




\section{Introduction} \label{sec1}

Change point analysis has found considerable interest in the last two decades
 because of its   numerous applications in  economics, climatology, engineering, hydrology, genomics,
 to mention just a  few.  Most of the recent results  are well documented in the reviews
by  \cite{auehor2013}, \cite{Jandhyala2013}, \cite{WoodallMontgomery2014}, \cite{Sharma2016}, \cite{ChakrabortiGraham2019}, and \cite{truong2020} among others.
In the simplest case one is interested in identifying  structural breaks  in a sequence of means $(\mu_i)_{i =1,\ldots ,n }$ 
of a possibly non-stationary time series  $(X_i)_{i =1,\ldots ,n }$. Often  the  data are modelled by a location scale model $X_i =  \mu (i /n)   + \varepsilon_i$ with a stationary error process $( \varepsilon_i)_{i =1,\ldots ,n }$ and a piecewise constant
 mean function $\mu : [0,1] \to \R$. A large amount of the literature on this problem refers to functions with exactly one change point
{\citep[see, e.\,g.][among others]{Priestley1969, Wolfe1984, Horvath1999}}
  but  more recently the problem of detecting multiple change points in a piecewise constant mean has also found considerable 
  attention 
 \citep[see, e.g.][among  many others]{frick2014,fryzlewicz2018a,detschvet2018,baranowski2019}. In  these cases  the null hypothesis 
 of no change point can be formulated as 
 \begin{align} \label{null}
H_0: \text{ there exists $\mu\in\R$ such that }\mu (t)  = \mu \text{ for all } t \in [0,1],
\end{align}
while, under the alternative, the authors make the assumption that the process is stationary on time spans with constant mean. This assumption simplifies
the statistical analysis of  structural breaks substantially.

While the assumption of a piecewise constant mean function is well justified in some applications
\citep[see for example][]{Aston2012, hotz2013, Cho2015, Kirch2015}, there
are also many other situations where  it is more realistic to assume that the function $\mu $ varies smoothly in the interval $[0,1]$. Typical examples include temperature data \citep[see, e.\,g.][]{karl1995,collins2000} or financial data (see, e.\,g., \citealp{Dette2014}). In these cases, it is more reasonable to assume that  the regression function changes smoothly 
and  one might be interested in deciding whether these changes  deviate  in some sense ``substantially'' from a given benchmark.
For example, if $\mu (0) $  denotes the initial mean at ``time'' $0$, it is often of interest whether the mean stays within a certain corridor of width $\Delta >0$, that is
\begin{equation} 
\label{h0initial} 
H_0:   d_{\infty} =  \sup_{t \in [0,1] }  | \mu (t) - \mu (0) | \leq \Delta  \quad\text{ vs. }\quad H_1:   
d_{\infty}  >  \Delta.
\end{equation}
Here $\Delta \geq 0 $ defines a pre-specified constant that is chosen according to the specific interests for the problem at hand. In case of rejection of $H_0$, a statistician may consider to subsequently  analyze the data with a time dependent 
 mean.   Note that 
 the null hypothesis in \eqref{null}  is obtained from \eqref{h0initial}  for $\Delta=0$. 
Other benchmarks could be
 used as well and it could also be reasonable to consider deviations on a sub-interval of $[0,1]$. For example, in climate research one might 
 be interested in significant deviations of  a  trend from an average trend in the past, and this problem  could be considered investigating the hypotheses
\begin{align} 
 \label{h0avl} 
H_0:     d_{\infty} =  \sup_{t \in [x_0,1] } \Big  | \mu (t) - \frac1{x_0}\int_{0}^{x_{0}}\mu (s)  \diff s \Big |  \leq \Delta  
~ \text{   vs.  } ~
H_1:   
d_{\infty} >  \Delta,
\end{align}
for some given fixed constant $x_0\in(0,1)$.
The consideration of  hypotheses of  the form  \eqref{h0initial}  or \eqref{h0avl} may also be motivated by the fact that the detection of structural breaks in the signal 
often results in an adaptation  of the statistical analysis (for example in forecasting).  
Because such an  analysis is usually performed ``locally'',  resulting estimators will have a smaller bias 
but a larger variance. However, if the changes in the signal are only weak, such an adaption might not be necessary
because a potential  decrease in bias might be overcompensated by an increase of variance.
  
 In this paper we develop statistical methodology to investigate hypotheses  
 of the form \eqref{h0initial}  and \eqref{h0avl}  (and further hypotheses of similar type)
 in a  location scale model with a  stationary error process and a smooth mean function.  
 Additionally, we also construct estimators for the first point in time where the null hypothesis is violated. 
 Note that this problem is related to  the economic design of control charts for quality control purposes
 which have their focus  on sequentially  detecting a change as quickly as possible after it occurs \citep[see, for example][]{champwood1987,woodmont1999}.
 However,  the focus of our approach here 
is on testing for the presence and identification  of the time  of a  change in a retrospective scenario. 
 Despite of its importance  - to  the best of our knowledge - not much work has been done   in this direction. The paper which is most similar in spirit to the problem considered here is the work
by \cite{dettewu2019}  who  define a change  in the mean of  a time series from its initial value as relevant if the amount of the change and the time period where the change is in place are reasonably large.
More precisely, these authors propose to test the hypothesis  that the Lebesgue measure of the set $\mathcal{M}_{\Delta} =\{ t \in [0,1] \colon   | \mu (t) -  \mu ({0}) | > \Delta  \} $ 
is smaller than a given value $c$, that is 
\[
\tilde H_{0}: \lambda (\mathcal{M}_{\Delta}) \leq c \quad \text{   vs.  }  \quad
\tilde H_1: \lambda (\mathcal{M}_{\Delta}) >   c.
\]
 In contrast to this, the formulation of  the hypotheses \eqref{h0initial} and \eqref{h0avl} defines a change as relevant whenever
the difference between the  mean function and the benchmark exceeds the threshold $\Delta$. The latter hypotheses are easier to interpret for  practitioners, however,
various challenging mathematical problems arise from this formulation.
While the Lebesgue measure of the set $\mathcal{M}_{\Delta} $ can be estimated by a mass excess approach and the corresponding statistic has 
an asymptotic normal distribution, such a simple limit distribution does not appear if one investigates  the  maximal deviation of the 
function $\mu $ from a  benchmark as formulated in the hypotheses \eqref{h0initial} or \eqref{h0avl}. 
More precisely, for the construction of a  test for these hypotheses 
we propose to  estimate  the  maximal deviation $d_\infty$ directly   and  to  
reject the null hypothesis for large values of its estimate.  In order to 
quantify the type~I and type II error, we investigate the asymptotic properties 
of the estimator, for any value $d_{\infty} \geq 0$, which depend sensitively on 
specific properties of the function $\mu$. 

In Section \ref{sec2},
we introduce a slightly more general version of the testing problems in 
\eqref{h0initial} and \eqref{h0avl} and propose an estimator for the respective 
maximum deviation. Several technical assumptions required 
for the asymptotic analysis are collected as well. 
 In  Section \ref{sec3},  we  show that the estimator (after appropriate standardization)  converges weakly to a non-degenerate limit distribution of the Gumbel type.
 The limit distribution  as well as the quantities used for standardization depend on the set of extremal points of the difference between the function $\mu$ and its benchmark, i.e., the set of points where the difference attains its sup-norm (see equation \eqref{extrem}  for a precise definition). The results in this section may already  be used to construct a simple consistent 
 and asymptotic level $\alpha$-test for the hypotheses \eqref{h0initial} or \eqref{h0avl}, but the test tends to be conservative if the Lebesgue measure of the set of extremal points is small.
In  Section \ref{sec4}, as a circumvention for the latter, we propose suitable estimators for the  set of extremal points,  prove that these are consistent and  use them  to develop  tests with a better approximation of the nominal level.  In Section~\ref{sec5}, we consider the estimation of the first point in time where the relevant change occurs.
 Section \ref{sec6} is then devoted to the detection of relevant change points in the signal to noise ratio of   a  
 location scale model with a non-stationary error process, while the finite sample properties of the proposed methodologies are investigated by means of a Monte Carlo simulation study in Section~\ref{sec7}. Finally, all proofs and technical details are deferred to an appendix.
 
		Throughout this paper, the symbol $\convw $ denotes weak convergence, and all convergences are for $n\to\infty$ if not mentioned otherwise.

\section{The general testing problem and mathematical preliminaries} \label{sec2}

We consider the common  location scale model
 \begin{align}\label{Lipschiss}
X_{i,n} =\mu(i/n)+\epsilon_{i}, \quad i=1,\dots,n,
\end{align}
 where $(\epsilon_{i})_{i\in \Z} $  denotes a stationary sequence of centered random variables and
$\mu : [0,1] \to \R$ is the unknown mean function.  
We are interested in detecting significant deviations  of the regression function $\mu$  on an interval $[x_{0},x_{1}] \subseteq [0,1]$ from a given benchmark. For this purpose,
we consider a real-valued functional $g$ on the space of all functions defined on the interval $[0,1]$ and define
the distance
\begin{equation}  \label{dinfty} 
 d_{\infty}:=  \sup_{t \in [x_0,x_{1}] } \big  | \mu (t) - g(\mu)  \big | 
 \end{equation}
 Note that this distance depends on the points $x_{0}, x_{1} \in [0,1] $ in the calculation of  the supremum and on the functional  $g$ used to define the benchmark, which is not reflected by our notation (as it will always be clear from the context).
 We are interested in the hypotheses
\begin{equation}
 \label{h0all} 
H_0:   d_{\infty}   \leq \Delta  \quad \text{   vs. }  \quad H_1:    d_{\infty} >  \Delta,
\end{equation}
where $\Delta>0$ is a given constant.
In particular, we obtain
\begin{compactitem}
\item the hypotheses \eqref{h0initial}  for the choice $g(\mu) = \mu (0)$, $x_{0}=0$ and $x_{1} =1$.
\item the hypotheses \eqref{h0avl}  for the choice $g(\mu) = (1/x_0) \int_{0}^{x_0}\mu (t) dt $ and $x_{1}  =1$.
\item the hypotheses of a relevant deviation from an overall mean by choosing $g(\mu) = \bar \mu = \int_{{0}}^{1}\mu (t) dt $, $x_{0}=0$ and $x_{1} =1$, that is
\begin{equation} 
 \label{h0full} 
H_0:    \sup_{t \in [0,1] } \big  | \mu (t) - \bar \mu \big |  \leq \Delta  \quad  \text{  vs. }  \quad H_1:   \sup_{t \in [0,1] } \big  | \mu (t) -  \bar \mu \big |   >  \Delta.
\end{equation}
\end{compactitem}
In order to estimate the maximum deviation $d_\infty$ defined in \eqref{dinfty} we use local linear regression. To be precise, 
let $K$ denote a kernel function (see Assumption~\ref{assump:kernel} below for details) and define $K_h( \cdot ) =K(\tfrac{\cdot}{h})$, for $h>0$. 
The \textit{local linear estimator} $\hat{\mu}_{h_n}$ with positive bandwidth $h_n =o(1)$ as $n\to\infty$ is  defined by the first coordinate of the minimizer 
\begin{equation} \label{eq:defLocLinEst}
 \big(\hat{\mu}_{h_n}(t),\widehat{\mu'}_{h_n}(t)\big) = \underset{b_0,b_1 \in \R }{\argmin} \sum_{i=1}^{n} \big\{ X_i-b_0-b_1(\tfrac{i}{n}-t)\big\}^2 K_{h_n}(\tfrac{i}{n}-t),
  \end{equation}
see, for example, \cite{FanGij1996}.  Throughout, we will use a bias corrected version of  the local linear estimator   adopting the Jackknife bias reduction technique proposed by \cite{schucany1977} which is defined by
\begin{equation}\label{ejackest} 
\tilde{\mu}_{h_n}(t)=2 \hat{\mu}_{h_n/\sqrt{2}}(t)-\hat{\mu}_{h_n}(t).
  \end{equation}
 A simple estimator for the benchmark $g( \mu) $ is now obtained by $\hat g_{n} =  g  (\tilde {\mu}_{h_n}) $, but we note that other estimators can be used as well for this purpose.
For example, if $g(\mu) = \int_{{0}}^{1}\mu (t) dt $,  the sample mean  $\bar{X}_n=\tfrac{1}{n}\sum_{i=1}^{n}X_{i,n}$ could be used as an alternative  (and simpler) estimator for $g(\mu)$.

If $\hat{g}_n$ is an appropriate estimator for the benchmark $g(\mu)$, a natural estimator for the maximum deviation  $d_\infty$ is given by
\begin{equation} 
\label{eq:estimator} 
 \hat{d}_{\infty,n} = \sup_{t\in I_n}| \tilde{\mu}_{h_n}(t)-\hat g_{n}|,
\end{equation} 
where we restrict the supremum to the interval $I_n=[x_0 \vee h_n , x_1 \wedge (1-h_n)]$ to account for boundary effects of $\tilde \mu_{h_n}$.
The null hypothesis in 
\eqref{h0all} is rejected for large values of the statistic $\hat d_{\infty, n} $. In order to define suitable critical values 
for this test  we  will investigate the weak convergence of an affinely standardized version of the statistic  $\hat d_{\infty, n}$ within the next section. 

\begin{remark} \label{rem:test0}
Note that the related problem of investigating the asymptotic properties of the sup-norm of the difference between a nonparametric estimator $\hat \mu$ and the true regression  function $\mu$  (or $\Exp[\hat \mu]$), i.e., $\| \hat \mu - \mu\|_\infty$ (or $\| \hat \mu - \Exp[\hat \mu ] \|_\infty$), have been investigated by several authors, mostly  in the case of independent observations \citep[see][and the references therein]{Johnston1982, Xia1998, Proksch2016}, 
but also for stationary data \citep[see][among others]{WuZhao2007,zhao2008}. The most prominent statistical application of results of this type concerns the construction of simultaneous asymptotic  confidence bands
for the mean function. Using the duality between confidence regions  and hypotheses testing,  these confidence bands yield a simple test for the hypotheses in \eqref{h0all}.
  In this remark we briefly discuss  this approach and explain why
it does not yield powerful tests in general.

Our starting point is a result of the above type in the case of a (possibly locally) stationary error process, which is in fact a by-product of this work derived in Section~\ref{sec8} below. 
In that section, we derive  the asymptotic distribution of  
\[
\sup_{t\in I_n} | \tilde{\mu}_{h_n}(t)-\hat g_{n} -\{ \mu(t) - g(\mu) \} |
\] 
(note that our assumptions allow for $\hat g_n=g=0$),  which may then be used to construct constants $c=c_{n,\alpha}$ not depending on unknown objects such that
\begin{equation}  \label{cb} 
	I_n(t) = [ \tilde{\mu}_{h_n}(t) -\hat g_{n}  -c_{n,\alpha}, \tilde{\mu}_{h_n}(t) -\hat g_{n}  +c_{n,\alpha}], \qquad {t\in I_n},
\end{equation} 
defines   an asymptotic simultaneous  $(1-\alpha)$ confidence band for the function $\mu- g(\mu)$,
that is, 
\begin{align*}
	&\phantom{{}={}}\lim_{n\to\infty} \Prob\big( \mu(t) - g(\mu) \in I_n(t)  \ \forall\, t \in I_n\big) \\
	&=
	\lim_{n\to\infty} \Prob( \sup_{t\in I_n }| \tilde{\mu}_{h_n}(t) -\hat g_{n} -\{ \mu(t) - g(\mu) \} |  \le c_{n,\alpha})  
	= 1-\alpha
	\end{align*}
	(see Remark~\ref{confband} for more  details, and \eqref{test1h0}  for a  precise  definition of $c_{n,\alpha}$).

Now using a well-known general relation between confidence intervals 
and  statistical tests (see, e.g., \citealp{aitchison1964}),  we may use \eqref{cb} to construct a simple asymptotic level $\alpha$-test. Indeed, some thoughts reveal that accepting the null hypotheses in  \eqref{h0all} whenever both the constant function $\Delta$ and the constant function $-\Delta$ is completely contained in the confidence band in \eqref{cb} is equivalent to rejecting  whenever
\begin{equation} \label{test0}
\hat d_{\infty, n}  > \Delta  + c_{n,\alpha}.
\end{equation}
Next, under $H_{0}$ in \eqref{h0all}, we have $d_{\infty,n} = \sup_{t \in I_n} | \mu(t) - g(\mu)| \le d_\infty \le \Delta $ and  
 \eqref{test0}  implies 
\[
\sup_{t\in I_n }| \tilde{\mu}_{h_n}(t) -\hat g_{n} -\{ \mu(t) - g(\mu) \} |
 \ge 
 \hat d_{\infty, n} - d_{\infty, n} >  c_{n,\alpha} + \Delta - \Delta =  c_{n,\alpha},
\]
which  gives 
\begin{multline}\label{conservative}
\limsup_{n \to \infty} \mathbb{P}
\big ( H_0 \mbox{ is rejected} 
 \big )  \\
 \le  
\limsup_{n \to \infty}  \mathbb{P}
\Big(  \sup_{t\in I_n }| \tilde{\mu}_{h_n}(t) -\hat g_{n} -\{ \mu(t) - g(\mu) \} |  > c_{n,\alpha} \Big)  =  \alpha.
\end{multline}
However, due to  the first inequality in \eqref{conservative},  the test in \eqref{test0} is conservative and has very low power for testing the hypotheses in
\eqref{h0all}, except in the case of classical hypotheses, that is $\Delta=0$.
In fact, in the latter case, we have equality in \eqref{conservative}  due to the fact   the null hypothesis is rather simple, only containing the constant regression functions. If $\Delta>0$ however, even the boundary of the null hypothesis (i.e., those regression functions $\mu$ with $d_\infty=\Delta$) is a far more complicated set, rendering the statistical problem substantially more difficult. 
\end{remark}

In light of the previous remark, it is the main goal of this work to derive alternative critical values for a test that is based on rejecting the null hypothesis for large values of $\hat d_{\infty, n}$. While those values will result in a substantially better approximation of the nominal level and better power properties, their derivation will also be more complicated compared to the above construction. More precisely, in Section \ref{sec3}  and \ref{sec4}, we will derive sequences $(a_n)_{n\in\N}$ and $(b_n)_{n\in\N}$ such that $a_n(\hat d_{\infty, n}-d_\infty)-b_n$ converges in distribution with a non-degenerate limit,
and it will turn out that both the standardizing sequences  and the limit distribution depend sensitively  on the function $d =d(t)= \mu(t) - g(\mu)$, even on the boundary of the null hypothesis.

In the remaining parts of this section, we collect regularity assumptions that are sufficient to derive the intended limit results.
We begin with a standard assumption regarding the kernel $K$ used in the local linear estimator.

\begin{assumption}\label{assump:kernel}
The kernel $K$ is symmetric, supported on the interval  $[-1,1]$, twice differentiable and satisfies 
$\int_{[-1,1]}K(x)\diff x=1$.
\end{assumption}

Next we define the dependence structure in model \eqref{Lipschiss}. For this purpose we
recall some basic definitions  on physical dependence measures of stationary processes  \citep{Wu2005}.
For $q\ge 1$, let  $\| X \|_{q,\Omega}  = \big ( \E |X|^q  \big )^{\scs 1/q}$ denote the $\mathcal L_q$-norm of a random variable $X$ defined on a probability space $(\Omega, \Ac, \Prob)$.
Let $\eta=(\eta_i)_{i\in\Z}$ be a sequence of independent identically distributed random variables and let $(\eta')=(\eta_i')_{i\in\Z}$ be an independent copy of $\eta$. Further, define $\Fc_i=( \ldots,\eta_{-2},\eta_{-1},\eta_0,\eta_1,\ldots,\eta_i)$ and $\Fc_i^*=(\ldots,\eta_{-2},\eta_{-1},\eta_0',\eta_1,\ldots,\eta_i)$. Let $G:\R^\N \to \R$ denote a possibly nonlinear filter such that $\eps_i=G(\Fc_i)$ and $\eps_i^*=G(\Fc_i^*)$ are properly defined random variables. 

The physical dependence measure of a  filter $G$ with $\|G(\Fc_0)\|_{q,\Omega}<\infty$ with respect to the norm $\|\cdot\|_{q,\Omega}$ is defined by
\begin{align} \label{eq:pdm}
	\delta_q(G,i)=\|G(\Fc_i)-G(\Fc_i^*)\|_{q,\Omega}, \qquad i \in \N.
	\end{align}
The quantity $\delta_q(G,i)$ can be regarded as a measure for the serial dependence at lag $i$ of $(\eps_{j})_{j\in\N}$.
It  plays a similar role as a mixing coefficient, yet it is easier to bound in many cases.

\begin{assumption} \label{assump:error}
	The error process $(\eps_{i})_{i\in\Z}$ in model \eqref{Lipschiss} is centered and has a representation $\eps_i=G(\Fc_i)$ with a filter $G$ such that the following conditions are satisfied:
	\begin{compactenum}[(i)]
	\item There exists $\chi\in(0,1)$ such that $\delta_4(G,i)=\Oc(\chi^i)$, as $i\to\infty$.
	\item The \textit{long-run variance} of $(\eps_{j})_{j\in\N}$, defined as 
	\begin{equation}\label{eq:defLRV} \sigma^2=\sum_{i=-\infty}^{\infty}\cov\big(G(\Fc_i),G(\Fc_0)\big) ,\end{equation}
	exists and is positive.
	\end{compactenum}
\end{assumption} 

We also need to impose a certain degree of  regularity on the function $d(t) = \mu (t) - g(\mu)$.  For that purpose, let  $\|f\|_\infty=\sup_{t\in[x_0,x_1]}|f(t)|$ denote   
the sup-norm of a function $f$ on the interval $ [x_0,x_1]$. Note that  $d_{\infty} =  \|d\|_\infty$ and that we again do not reflect the dependence on $x_{0}$ and $x_{1}$ in our notation. 
The points where  $|d| $ attains its sup-norm are  called 
extremal points and  the corresponding set of extremal points is  denoted by $\Ec$. Note that we have
\begin{equation}  \label{extrem}
\Ec=\Ec^+\cup \Ec^-,
\end{equation} 
 where
\[ 
\Ec^{+}= \{t\in[x_0,x_1]:  d(t)= \|d\|_\infty \},\quad \Ec^{-}= \{t\in[x_0,x_1]:  d(t)=  - \|d\|_\infty \}.
\]
These sets depend on the function $d$ and on $x_0, x_1$, which is not reflected in our notation as it will always be clear from the context.
If the function $d$ is continuous, the sets $\Ec,\Ec^+$ and $\Ec^-$ are compact.  Moreover, unless $\|d\|_\infty = 0$, $\Ec^+$ and $\Ec^-$ are disjoint.

\begin{assumption}\label{assump:mu}~
\begin{compactenum}[(i)] \item
The function $\mu$ is twice differentiable with Lipschitz continuous second derivative.

\item There exists a constant $\gamma>0$ such that $d = \mu - g(\mu) $ is concave (convex) on $U_\gamma(t) := \{s\in[x_0, x_1] \colon |s-t|< \gamma \} $, for any $t\in\Ec^+$ ($t\in\Ec^-$).
\end{compactenum}
\end{assumption}

\begin{remark}\label{rem:concavity}
Assumption \ref{assump:mu}(ii) is made to avoid functions with an irregular behaviour at the extremal points. In this remark we give some more explanation for it.
\begin{compactenum}[(i)] 
		\item[(i)] If the function $d$ is  continuously differentiable, Assumption \ref{assump:mu}(ii) is satisfied provided its derivative $d'$ is decreasing on $U_\gamma(t)$, for $t\in\Ec^+$, and 
		increasing on $U_\gamma(t)$, for $t\in\Ec^-$.	In particular, if $d$ is twice differentiable with Lipschitz continuous second derivative,
		and if $\phi = \inf_{t\in\Ec} |d''(t)|>0$, the constant  $\gamma$ can be chosen as $\phi /(2L_2)$, where  $L_{2}	$  is the Lipschitz constant of $d^{\prime\prime}$.  
		In this case  it follows that $\lambda(\Ec)=0$, where $\lambda$ denotes the Lebesgue measure. 
		
			\item[(ii)] 
		For any interval $[t_1,t_2]\subset\Ec^+$ with $t_1,t_2\in\partial \Ec^+$, by concavity, $d$ is strictly increasing on 
		the interval 
		$(t_1-\gamma, t_1]$ and
		strictly decreasing on $[t_2,t_2+\gamma)$. Consequently, for any $s\in (t_1-\gamma, t_1)\cup (t_2,t_2+\gamma)$, it follows that $s\notin\Ec^+$. Analogously $s\in (t_1-\gamma,t_1)\cup (t_2,t_2+\gamma)$ implies that $s\notin \Ec^-$, for any interval $[t_1,t_2]\subset \Ec^-$ with $t_1,t_2\in\partial \Ec^-$. Hence, it follows from Assumption \ref{assump:mu}(ii), that $\Ec$ is 
		a  finite union of at most $\lfloor (2\gamma)^{-1}\rfloor$ intervals and single points.
	\end{compactenum}
\end{remark}

\begin{assumption}\label{assump:ghat}
	The estimator $\hat{g}_n$ of the functional $g(\mu)$ satisfies 
	\[
		|\hat{g}_n-g(\mu)|=o_\pr\Big(\tfrac{1}{\sqrt{nh_n | \log(h_n )| }}\Big), \qquad n \to\infty,
		\]
		where $h_n$ denotes the bandwidth parameter used for the estimator in \eqref{eq:defLocLinEst}.
\end{assumption}

\begin{remark}
	Assumption \ref{assump:ghat} is not very strong and satisfied for many common estimators for $g(\mu)$. Exemplary, we consider the situations in \eqref{h0initial}, \eqref{h0avl} and \eqref{h0full}.
	\begin{compactenum}[(i)] 
		\item[(i)] If Assumptions \ref{assump:kernel}, \ref{assump:error} and \ref{assump:mu}(i) are met, it follows from Lemma~C.2 in the supplementary material of \cite{dettewu2019} that $$|\tilde{\mu}_{\tilde h_n}(0)-\mu(0)|=\Oc_\pr\Big(\tfrac{1}{\sqrt{n\tilde h_n}}\Big)+\Oc\Big(\tilde h_n^3+\tfrac{1}{n\tilde h_n}\Big),$$
		for any bandwidth $\tilde h_n > 0$ with $\tilde h_n\to 0$ and $n\tilde h_n\to \infty$. Thus, Assumption \ref{assump:ghat} holds for $g(\mu)=\mu(0)$ with $\hat{g}_n = \tilde{\mu}_{h_n |\log(h_n)|^2}(0)$, provided that $nh_n^7 | \log h_n|^{12} = o(1)$.
			\item[(ii)] Similarly, if $g(\mu)$ is an integral as considered in \eqref{h0avl} and \eqref{h0full}, define $\hat{g}_n=\bar{X}_n(x_0)=\tfrac{1}{\lfloor x_0 n\rfloor}\sum_{i=1}^{\lfloor x_0 n\rfloor} X_{i,n}$. Then,	
		\begin{align*}
		&\phantom{{}={}} \bigg|\bar{X}_n(x_0)-\frac1{x_0}\int_0^{x_0}\mu(t)\diff t\bigg| \\
		&= \bigg|\frac{1}{\lfloor x_0 n\rfloor}\sum_{i=1}^{\lfloor x_0 n\rfloor}\eps_i + \frac{1}{\lfloor x_0 n\rfloor}\sum_{i=1}^{\lfloor x_0 n\rfloor}\mu\big(\tfrac{i}{n}\big)- \frac1{x_0} \int_0^{x_0} \mu(t)\diff t\bigg| 
		= \Oc_\pr(n^{-1/2})
		\end{align*}
		and Assumption \ref{assump:ghat} holds for $g(\mu)=\frac1{x_0}\int_0^{x_0}\mu(t)\diff t$, provided $h_n=o(1)$ as $n\to\infty$.	\end{compactenum}
\end{remark}

	\section{Weak convergence of $\widehat d_{\infty, n}$ and a first simple test}  \label{sec3}

	In this section we will
	derive  the limit distribution of the statistic $\hat d_{\infty, n}$ (after appropriate 
	standardization), see Theorem~\ref{thm:Dn}. The result may be used to construct a simple test for the hypotheses  in \eqref{h0all}, for any fixed $\Delta\ge0$.	However, depending on the data-generating process, the test may be quite conservative whence we continue in Section \ref{sec4} by proposing  a less conservative test based on estimating the set of extremal points $\mathcal{E} $ defined  in \eqref{extrem}.

The convergence rate of the statistic $\hat d_{\infty,n}$ as defined in \eqref{eq:estimator} depends crucially on the scaling sequence $\ell_n$ defined as
			\begin{equation} \label{eq:ln} 
		\ell_n=\sqrt{2\log(\tfrac{\Lambda_{K} \lambda(I_n) }{2\pi h_n})} \sim \sqrt{ 2  \cdot |\log (h_n)|},
		\end{equation}	
		where $I_n=[x_0 \vee h_n, x_1 \wedge (1-h_n)]$ and the constant  $\Lambda_{K}$ is defined as
		\begin{equation} \label{lamK}
		\Lambda_{K} = {  \|(K^*)^{\prime }\|_2  \over \|K^*\|_2} 
		\end{equation}	
		($\|f\|_2$ denotes the $L^2$-norm of some function $f$), with the function $K^{*}$ given by  
		\begin{equation}
		\label{eq:defassump:kernel}
		K^*(x)= 2\sqrt{2}K(\sqrt{2}x)-K(x).
		\end{equation}	
		Next, recall the Gumbel distribution $\Gum_a$ with location parameter $a \in \R$, defined through its c.d.f.
		\[
		\Gum_a((-\infty, x]) = \exp[-\exp\{-(x-a)\}], \quad x \in \R.
		\] 
Finally, recall  that $\hat d_{\infty,n} = \sup_{t \in I_n} |\tilde{\mu}_{h_n}(t)-\hat{g}_n|$ and let $d_{\infty,n} = \sup_{t \in I_n} |\mu(t)-g(\mu)|$.

	\begin{theorem}\label{thm:Dn}
Suppose that  Assumptions \ref{assump:kernel}, \ref{assump:error}, \ref{assump:mu}(i) and  \ref{assump:ghat}  hold. Further, assume that the bandwidth satisfies
 \[
 h_n\to 0, \quad nh_n\to\infty, \quad nh_n^{7} |\log(h_n)|\to 0, \quad 
 \limsup_{{n\to \infty} }  \tfrac{|\log(h_n)| \log^4 n}{n^{1/2}h_n}<\infty.
 \]
 \begin{compactenum}[(1)]
\item[(1)] If $\lambda(\Ec)>0$ and $d_\infty>0$, then 
\[ 
	\frac{\ell_n \sqrt{nh_n}}{\sigma\|K^*\|_2} ( \hat d_{\infty,n}  -  d_{\infty,n}) -\ell_n^2 
	\convw  
	\Gum_{\log\{ \lambda(\Ec)/(x_1-x_0) \} }.
	 \]
\item[(2)]	 If $d_\infty=0$ (which implies $\lambda(\Ec)=x_1-x_0)$, then 
	\[ 
	\frac{\ell_n \sqrt{nh_n}}{\sigma\|K^*\|_2} ( \hat d_{\infty,n}  -  d_{\infty,n}) -\ell_n^2 
\convw  \Gum_{\log(2)}.
	 \]
	\item[(3)] If $\lambda(\Ec)=0$ and if additionally Assumption~\ref{assump:mu}(ii) is met, then 
		\[ 
	\frac{\ell_n \sqrt{nh_n}}{\sigma\|K^*\|_2} ( \hat d_{\infty,n}  -  d_{\infty,n}) -\ell_n^2 
 \convw  - \infty.
	 \]
\end{compactenum}
	\end{theorem}
	
Note that $-\infty \ll_S \Gum_{\log\{ \lambda(\Ec)/(x_1-x_0) \}} \ll_S \Gum_0 \ll_S \Gum_{\log(2)}$, where $\ll_S$ denotes the (first order) stochastic dominance ordering.  Subsequently, this will be used to construct suitable critical values for a first simple test of \eqref{h0all}. Further note that the results of the theorem continue to hold if $\ell_n$ is replaced by
		\begin{equation} \label{eq:ln2} 
		\ell_n'=\sqrt{2\log(\tfrac{\Lambda_{K} (x_1-x_0) }{2\pi h_n})}.
		\end{equation}

In general, applications of Theorem~\ref{thm:Dn} to the construction of inferential methodology  require an estimator for the long-run variance $\sigma^2$. For this purpose, we may for instance follow \cite{WuZhao2007}: define the partial sums $S_{j,k}=\sum_{i=j}^{k} X_{i,n}$ and let
\begin{equation} \label{lrvest}
\hat{\sigma}^2=\frac{1}{\lfloor n/m_n\rfloor -1}\sum_{j=1}^{\lfloor n/m_n\rfloor-1}\frac{(S_{(j-1)m_n+1,jm_n}-S_{jm_n+1,(j+1)m_n})^2}{2m_n},
\end{equation} 
where $m_n$ denotes some integer sequence proportional to $n^{1/3}$. If Assumptions \ref{assump:kernel}, \ref{assump:error}  and \ref{assump:mu}(i) are satisfied, it follows from Theorem 3 
in the last-named reference
that 
\begin{equation}\label{eq:estLRV}\hat{\sigma}^2=\sigma^2+\Oc_\pr(n^{-1/3}).\end{equation}

Together with the latter result, Theorem~\ref{thm:Dn} allows for the construction of a first simple test for the hypotheses  in \eqref{h0all}. The cases $\Delta=0$ and $\Delta>0$ need to be treated separately. First, in the case $\Delta=0$, the only possible limiting distribution under the null is the Gumbel distribution $\Gum_{\log(2)}$, whence we propose to reject the null hypothesis
\begin{equation}
\label{classh0}
H_{0}: d_{\infty} =0 \quad \mbox{  vs. }  \quad H_{1}:  d_{\infty} >0 
 \end{equation}
whenever
\begin{equation}
\label{test1h0}
\hat d_{\infty, n} > (q_{\log(2),1-\alpha}+\ell_n^2)\frac{\hat{\sigma}\|K^*\|_2}{\sqrt{nh_n}\ell_n} =: c_{n,\alpha}, 
 \end{equation}
 where 
  $q_{a, \beta}$ denotes the $\beta$-quantile of the Gumbel distribution $\Gum_a$ with location parameter $a$. In the case $\Delta>0$, all possible limiting distributions on the boundary of the null hypothesis (i.e., $d_\infty=\Delta$) are stochastically dominated by the Gumbel distribution $\Gum_0$ (which is in fact attained for models with $\lambda(\Ec)=x_1-x_0$). As a consequence, we propose to reject the null hypothesis 
\begin{equation}
\label{relh0}
H_{0}: d_{\infty} \leq  \Delta  \quad \mbox{  vs. }  \quad H_{1}:  d_{\infty} >\Delta,
 \end{equation}
whenever
\begin{equation}
\label{test1rel}
\hat d_{\infty, n} > (q_{0,1-\alpha}+\ell_n^2)\frac{\hat{\sigma}\|K^*\|_2}{\sqrt{nh_n}\ell_n}+\Delta.
 \end{equation}
Note that this decision rule   has a similar structure as the   test  \eqref{test0} derived from the simultaneous confidence band for $d$, whose critical value are based on using $q_{\log 2,1-\alpha  }$
instead of $q_{0,1-\alpha  }$; see equation \eqref{cnalpha} for the definition of $c_{n,\alpha}$  in   \eqref{test0}.
Consequently, as  $q_{0,1-\alpha }   < q_{\log 2,1-\alpha }  $ the test \eqref{test1rel} is more powerful than  the test  \eqref{test0}.

	\begin{corollary}  \label{FloZirkus}
	Let Assumptions \ref{assump:kernel}, \ref{assump:error}, \ref{assump:mu} and  \ref{assump:ghat}  be met.
	\begin{compactenum}[(1)]
	\item The decision rule \eqref{test1h0} defines a consistent and asymptotic level $\alpha$-test for the 
	the hypotheses in \eqref{classh0}.
	\item[(2)] The decision rule \eqref{test1rel} defines a consistent and asymptotic level $\alpha$-test for the 
	the hypotheses in \eqref{relh0}, where $\Delta >0$.
	\end{compactenum}	
	\end{corollary}

It follows from the proof of Corollary \ref{FloZirkus} in Section \ref{subsec:proof3} that the test defined by  \eqref{test1rel}
for the relevant hypotheses \eqref{relh0} 
is  conservative for those models on the boundary of the null hypothesis for which  $\lambda(\Ec)$ is strictly smaller than $x_1-x_0$.
A  heuristic argument for this fact is that  the test does not use the quantiles of the $\Gum_{\log\{ \lambda(\Ec)/(x_1-x_0) \}}$-distribution 
in Theorem~\ref{thm:Dn}(1), but only  the quantiles of the (stochastically dominating) distribution $\Gum_{0}$. In  the following section we will address this 
 problem by   estimating  the unknown set $\Ec$ and develop 
a test which does not have this drawback.

	\section{An improved test based on estimation of the set of extremal points} \label{sec4}

The test in \eqref{test1rel} for the hypotheses in \eqref{relh0} suffers from the fact that the models on the boundary of the null hypothesis may have different limiting distributions, whence only the stochastically largest may be used to construct valid critical values. A workaround proposed in this section consists of estimating
the Lebesgue measure  $\lambda(\Ec)$ of the set of extremal points  first and then using a slightly different standardization for $\hat d_{\infty, n}$, resulting in the same limit distribution for all models with $d_\infty=\Delta>0$ and $\lambda(\Ec)>0$. Moreover, we will provide 
 an approximation result in terms of i.i.d.\ Gaussian variables, which may be used as yet another alternative to obtain valid critical values that are non-conservative for  models on the boundary of the null hypothesis.

For the estimation of $\lambda(\Ec)$, recall the definition of $I_n=[x_0 \vee h_n , x_1 \wedge (1-h_n)]$ and define, for some positive sequence $(\rho_n)_{n\in\N}=o(1)$, the set $\Ec_n=\Ec_n^+\cup \Ec_n^-$, where
\begin{equation}\label{eq:defE}
  \Ec_n^{\pm}=\Big \{ t\in I_n: d_{\infty,n} \mp d(t)\leq \rho_n \Big \}. 
 \end{equation}
Replacing $d(t)=\mu(t)-g(\mu)$ by $\hat d_n(t)=\tilde \mu_{h_n}(t)-\hat g$, we obtain the set-valued estimator $\hat{\Ec}_n = \hat{\Ec}_n^+ \cup \hat{\Ec}_n^-$, where
\[
 \hat{\Ec}_n^\pm =  \Big \{ t\in I_n: \hat d_{\infty, n} \mp \hat{d}_n(t)\leq \rho_n  \Big\}. 
 \]
 This estimator is consistent in the following sense.

\begin{theorem}\label{thm:EnEstimator} 
Let $\rho_n=o(1)$ be a positive sequence such that 
$\rho_n^2 nh_n / |\log (h_n)|\to\infty$. 
Further, suppose that Assumptions \ref{assump:kernel}, \ref{assump:error}, \ref{assump:mu}(i) and  \ref{assump:ghat}  are satisfied. If $\lambda(\Ec)=0$, let additionally Assumption \ref{assump:mu}(ii) and $\rho_n h_n^{-2} \to \infty$ be satisfied. Then, 
	\[ \frac{\lambda(\Ec_n)}{\lambda(\hat{\Ec}_n)}\convp 1. \]
\end{theorem}

Next, we present a variation of Theorem~\ref{thm:Dn} which is based on the alternative scaling sequence $\ell_n(\Ec_n)$ (instead of $\ell_n=\ell_n([x_0,x_1])$ as used in Theorem~\ref{thm:Dn}), where 
\begin{equation} \label{lnA}
	\ell_n(\Ac)=\sqrt{2\log\big(\tfrac{\Lambda_{K}\lambda(\Ac \cap I_n)}{2\pi h_n}\big)}, \qquad \Ac \in \Bc([x_0,x_1])
	\end{equation}	
	and $\Bc([x_0,x_1])$ denote the Borel subsets of $[x_0,x_1]$.

\begin{theorem}\label{conv:DnEn}
Suppose that Assumptions \ref{assump:kernel}, \ref{assump:error}, \ref{assump:mu} and  \ref{assump:ghat}  are satisfied. If $\lambda(\Ec)=0$, let additionally $\rho^{1/2}_{n}h_n^{-1}\to \infty$. Let $(V_{i})_{{i\in \N} }$ denote an i.i.d.\  sequence of standard normally distributed random variables and define
\begin{align}\label{eq:Gn}
	 	G_{n,1} &=\frac{\ell_n(\Ec_n)  \sqrt{nh_n}}
		 {\|K^*\|_{2}} \sup_{t\in\Ec_n} \Big\{ \frac1{nh_n} \sum_{i=1}^{n}V_iK_{h_n}^*\big(\tfrac{i}{n}-t\big)  \Big\} 
		-\ell_n^2(\Ec_n),  \\
		\label{eq:Gn2}
		G_{n,2} &=\frac{\ell_n(\Ec_n)  \sqrt{nh_n}}
		 {\|K^*\|_{2}} \sup_{t\in\Ec_n} \Big| \frac1{nh_n} \sum_{i=1}^{n}V_iK_{h_n}^*\big(\tfrac{i}{n}-t\big)  \Big| -\ell_n^2(\Ec_n).
	 	\end{align}
If $d_\infty>0$, then 
\[  
\pr\Big(\tfrac{\ell_n(\Ec_n)\sqrt{nh_n}}{\sigma\|K^*\|_2} \big( \hat d_{\infty, n} - d_{\infty, n}\big)-\ell_n^2(\Ec_n)\leq x\Big)\geq \pr(G_{n,1}\leq x)+o(1), 
\]
with equality if $\lambda(\Ec)>0$.  If $d_\infty=0$, then
\[  
\pr\Big(\tfrac{\ell_n(\Ec_n)\sqrt{nh_n}}{\sigma\|K^*\|_2} \big( \hat d_{\infty, n} - d_{\infty, n}\big)-\ell_n^2(\Ec_n)\leq x\Big) = \pr(G_{n,2} \leq x)+o(1).
\]	
Moreover,
		\[
		G_{n,1} \convw \Gum_0, \qquad G_{n,2} \convw \Gum_{\log(2)}
		\]
\end{theorem}

We can now use the previous two theorems to define consistent asymptotic level $\alpha$ tests for the hypotheses \eqref{classh0} and \eqref{relh0}, based on a rescaling of
the statistic   $\hat d_{\infty, n}$ which makes use of the estimator $\hat \Ec_n$.  To be precise, we propose  to reject the null hypotheses in \eqref{classh0} and \eqref{relh0} if
\begin{equation} 
\label{test:basedOnEstimator} 
\sup_{t\in I_n}|\hat d_n(t)| > \{ q_{a,1-\alpha}+\ell_n^2(\hat\Ec_n) \} \frac{\hat{\sigma}\|K^*\|_2}{\sqrt{nh_n}\ell_n(\hat\Ec_n)}+\Delta, 
\end{equation}
where $\hat \sigma^{2}$ denotes the
 estimator of the long-run variance defined in \eqref{lrvest} 
 and
$q_{a, 1-\alpha}$ denotes the $(1-\alpha)$-quantile of the Gumbel distribution $\Gum_a$ with location parameter $a=0$ if $\Delta>0$ and $a=\log(2)$ if $\Delta=0$. 

	\begin{corollary}  \label{test2}
Under the assumptions of Theorem \ref{conv:DnEn}, the test \eqref{test:basedOnEstimator} is consistent and has asymptotic level $\alpha$ for the hypotheses in \eqref{classh0} and \eqref{relh0}.
	\end{corollary}

 From Theorem \ref{thm:EnEstimator}, Theorem \ref{conv:DnEn} and \eqref{eq:estLRV}, we also obtain that, if $d_\infty>0$,
  \begin{equation*}
 \pr\Big(\tfrac{\ell_n(\hat \Ec_n)\sqrt{nh_n}}{\hat \sigma\|K^*\|_2} \big( \hat d_{\infty, n} - d_{\infty, n}\big)-\ell_n^2(\hat \Ec_n)\leq x\Big)\geq \pr(G_{n,1}\leq x)+o(1),
\end{equation*}
with $G_{n,1}$ as defined in \eqref{eq:Gn}. Hence, while the test in  
\eqref{test:basedOnEstimator} (for the case $\Delta>0$) was based on  using the quantiles $q_{0,1-\alpha}$ of the limiting distribution $\Gum_0$ of $G_{n,1}$, we may   alternatively use the quantiles $q_{\scs \hat G_{n,1},1-\alpha}$ of the distribution of  $\hat G_{n,1}$ (or $q_{\scs \hat G_{n,2}, 1-\alpha}$ in case $\Delta=0$), where $\hat G_{n,j}$ is defined analogously to $G_{n,j}$ in \eqref{eq:Gn} and \eqref{eq:Gn2},  but with $\Ec_n$ replaced by $\hat \Ec_n$.
 Note that these quantiles may easily be simulated up to an arbitrary precision. More precisely, we propose to reject the null hypotheses in \eqref{classh0} and \eqref{relh0} if
  \begin{equation}\label{test:approxQuantiles}\sup_{t\in I_n}|\hat{d}_n(t)| > \{ q_{\hat G_{n,j},1-\alpha}+\ell_n^2(\hat\Ec_n) \}\frac{\hat{\sigma}\|K^*\|_2}{\sqrt{nh_n}\ell_n(\hat\Ec_n)}+\Delta, 
  \end{equation}
  where $j=1$ if $\Delta>0$ and $j=2$ if $\Delta=0$.  It can be shown by similar arguments as given in the proof of Corollary \ref{test2} that this test is consistent and has asymptotic level $\alpha$. 
The numerical results in Section \ref{sec7} suggest that the test in  \eqref{test:approxQuantiles} exhibits better finite sample properties than the test in  \eqref{test:basedOnEstimator}. This phenomenon may be heuristically explained by the fact that the convergence rate of a maximum of independent normal random variables to the Gumbel distribution is rather slow.

\section{Estimating the time of the first relevant deviation} \label{sec5} 

	The aim of this section is to develop an estimator for the first relevant deviation
	 \[
	 t^* = \inf\{ t\in[x_0,x_1] : |d(t)|\geq \Delta \},
	 \] 
	 where we use the convention that $\inf(\emptyset)=+\infty$.
	First note that by continuity of $d$ the point $t^*$ can be represented as 
	\[
	t^*=x_0+\int_{x_0}^{x_1} \id \Big (\max_{t\in[x_0,s]}|d(t)|<\Delta \Big )\diff s + \infty \cdot \id \Big (d_\infty <\Delta \Big ).
	\]
	Obviously, the properties of any estimator will depend on the smoothness of the function $d$ at the point $t^*$. To capture the degree of smoothness, assume that there exist constants $\kappa>0$ and $c_\kappa>0$ such that
	\begin{equation}\label{eq:smoothnessMu}
	\lim_{s \uparrow t^*}\frac{|d(t^*)-d(s)|}{(t^*-s)^\kappa} =  c_\kappa.
	\end{equation}
	Note that $\kappa=1$ if the function $d$ is differentiable at the point  $t^*$ with non-vanishing derivative.
		
		\begin{theorem}\label{thm:estFirstExceedance}
	Let  Assumptions \ref{assump:kernel}, \ref{assump:error}, \ref{assump:mu} and \eqref{eq:smoothnessMu} be satisfied and let $\delta_n $ denote a positive sequence such that $\delta_n\to 0$
	and  $\liminf_{n\to  \infty}\tfrac{\sqrt{nh_n}\delta_n}{\sigma \|K^*\|_2}-  \ell_n  >0$. If $t^* \in [x_0, x_1]$, then the estimator
$$
	\hat{t}^* =(h_n\vee x_0)+ \int_{I_n} \id \Big (\max_{t\in[h_n \vee x_0,s]}|\hat{d}_n(t) |<\Delta - \delta_n \Big )\diff s  
	+ \infty \cdot  \id \Big (\hat d_{\infty,n} <\Delta - \delta_n \Big )
$$
	satisfies
	$$\hat{t}^* = t^* + \Oc_\pr\Big(\big(\tfrac{|\log(h_n)|^{1/2} }{\sqrt{nh_n}}+\delta_n\big)^{1/\kappa}\vee h_n\Big) =  t^* + o_\Prob(1).$$
	If $t^*=\infty$, then $\Prob(\hat t^*<\infty)=o_\Prob(1)$.
	\end{theorem}

	\section{Extension to non-stationary error processes} \label{sec6}

In this section, we extend the theory developed in the previous sections to the model
\begin{equation}\label{modneu}
X_{i,n} = \mu(i/n)+\eps_{i,n}, \qquad i=1, \dots, n,
\end{equation}
with a triangular array of centered but possibly non-stationary (e.g., heteroscedastic) errors $\{ \eps_{i,n}|1\leq i\leq n \}_{n\in\N}$. 
For this purpose, recall the basic definitions  on physical dependence measures stated before Assumption~\ref{assump:error}.
A  triangular array $\{ \eps_{i,n}|1\leq i\leq n \}_{n\in\N}$ of random variables is called \textit{locally stationary} if there exists a possibly non-linear filter $G:[0,1] \times \R^\N \to \R$ which is continuous in its first argument such that $\eps_{i,n}=G(i/n,\Fc_i)$ for all $i=1,\ldots,n$ and $n\in\N$.  The physical dependence measure defined in \eqref{eq:pdm} may be extended to a filter $G=G(\cdot,\cdot)$ with $\sup_{t\in[0,1]}\|G(t,\Fc_0)\|_{q,\Omega}<\infty$ by
\[ \delta_q(G,i)=\sup_{t\in[0,1]}\|G(t,\Fc_i)-G(t,\Fc_i^*)\|_{q,\Omega}, \quad i \in \N. \]

\begin{assumption} \label{assump:errorLS}
	The triangular array $\{(\eps_{i,n})_{1\leq i\leq n}\}_{n\in\N}$ in model \eqref{modneu}  is centered and locally stationary with a filter function $G$ that satisfies $\sup_{t\in[0,1]}\|G(t,\Fc_0)\|_{4,\Omega}<\infty$. Moreover, the following conditions are met:
	\begin{compactenum}[(1)]
		\item There is a constant $\chi\in(0,1)$ such that $\delta_4(G,i)=\Oc(\chi^i)$, as $i\to\infty$.
		\item The filter $G$ is Lipschitz continuous with respect to $\|\cdot\|_{4,\Omega}$,  that is  
		\[ \sup_{0\leq s<t\leq 1}\|G(t,\Fc_0)-G(s,\Fc_0)\|_{4,\Omega}/|t-s|<\infty. \]
		\item The \textit{long-run variance function} of $\{(\eps_{i,n})_{1\leq i\leq n}\}_{n\in\N}$, defined as 
		\[ \sigma^2(t)=\sum_{i=-\infty}^{\infty}\cov\big(G(t,\Fc_i),G(t,\Fc_0)\big) , \qquad t \in [0,1], \]
		is Lipschitz continuous and bounded away from zero, i.e., $\sigma_{\min}^2:=\inf_{t\in[0,1]}\sigma^2(t)>0$.
	\end{compactenum}
\end{assumption} 

As  $ \sigma^2(t)$ can be interpreted as  the (asymptotic) variance of a local mean of the data in a neighbourhood of  the point $t \in [0,1]$, it is reasonable to compare deviations relative to this local  noise. 
More precisely, consider a real-valued functional $g(\mu,\sigma)$ depending on the mean  $\mu$ and on the square root of the  long-run variance function $\sigma=\sqrt{\sigma^2}$, and define the distance
 \[
 d_\infty^\sigma:= \sup_{t\in[x_0,x_1]}\big|\mu(t)/\sigma(t)-g(\mu,\sigma) \big|. 
 \]
 We are interested in the hypotheses
\begin{equation}\label{h0_LS}
H_0: d_\infty^\sigma\leq \Delta \quad \text{   vs. } \quad H_1:d_\infty^\sigma >\Delta,
\end{equation} 
for some $\Delta \geq 0$. As in the stationary case, we need some regularity of the function 
\[
d^\sigma(t) = \mu (t)/\sigma(t) - g(\mu,\sigma), \qquad t \in [0,1],
\] 
and a suitable estimator for $g(\mu,\sigma)$.
Roughly speaking, the next two assumptions correspond to the assumptions in Section \ref{sec2}, where the function $\mu$ is replaced by $\mu/\sigma$.

\begin{assumption}\label{assump:ghatG}
	$\hat{g}_n$ is an estimator of $g(\mu,\sigma)$ such that $$|\hat{g}_n-g(\mu,\sigma)|=o_\pr\Big(\tfrac{1}{\sqrt{nh_n| \log(h_n)|}}\Big).$$
\end{assumption}

\begin{assumption}\label{assump:muLS}
~
\begin{compactenum}[(i)] 
\item
The functions $\mu$ and $\sigma$ are  twice differentiable with Lipschitz continuous second derivatives. 

	\item Define $\Ec^\sigma=\Ec^{+}\cup \Ec^-$, where
	\[ \Ec^{\pm} = \{t\in[x_0,x_1]: \pm d^\sigma(t)= \|d^\sigma\|_\infty \}. \]
	There exists a constant $\gamma>0$ such that $d^\sigma = \mu/\sigma - g(\mu,\sigma) $ is concave  (convex) on $U_\gamma(t) := \{s \in [0,1] \colon |s-t|< \gamma \} $, for any $t\in\Ec^+$  ($t\in\Ec^-$).
	\end{compactenum}
\end{assumption}

Further, we will need an estimator of the (time-dependent) long-run variance. For this purpose we follow \cite{dettewu2019} and define the partial sums $S_{j,k}=\sum_{i=j}^{k} X_{i,n}$ and the weight function 
$
\omega_{\tau_n}(t,j)=K_{\tau_n}\big(j/n-t\big)/\{\sum_{i=1}^{n}K_{\tau_n}\big(i/n-t\big)\},
$
for some positive bandwidth sequence $\tau_n=o(1)$. For some integer sequence $(m_n)_{n\in\N}$ with  $m_n \to \infty$ and $m_n \ll n$, define
\begin{equation}\label{eq:defLocLRVest} 
\hat{\sigma}^2(t)= \hat{\sigma}_{\tau_n, m_n}^2(t)=\sum_{j=1}^{n}\omega_{\tau_n}(t,j)\frac{(S_{j-m_n+1,j}-S_{j+1,j+m_n})^2}{2m_n}, 
\end{equation}
for $t\in[m_n/n, 1-m_n/n]$.
Extend this estimator to the whole interval $[0,1]$ by the definition $\hat{\sigma}^2(t)=\hat{\sigma}^2(m_n/n)$, for $t\in[0,m_n/n)$, and $\hat{\sigma}^2(t)=\hat{\sigma}^2(1-m_n/n)$, for $t\in(1-m_n/n,1]$. The following result specifies the convergence rate of $\hat\sigma^2$, and is proved in the Appendix.

\begin{theorem}\label{locstat}
	Let Assumption \ref{assump:errorLS} and \ref{assump:muLS}(i) be satisfied, and  assume that the smoothing parameter sequences $\tau_n>0$ and $m_n\in\N$ satisfy 
\[
\tau_n\to 0, \quad 
m_n\to \infty, \quad 
\tfrac{m_n^{1/4}}{\sqrt{n}\tau_n}\to 0, \quad
\frac{m_n^{5/2}}{n}\to 0.
\] 
Then,  with $\gamma_n=\tau_n+m_n/n$,
	\begin{equation*}
	\sup_{t\in[\gamma_n,1-\gamma_n]}|\hat{\sigma}^2(t)-\sigma^2(t)|=\Oc_\pr\Big(\frac{m_n^{1/4}}{\sqrt{n}\tau_n}+\frac{1}{m_n}+\tau_n^2+\frac{m_n^{5/2}}{n}\Big) = o_\pr(1).
	\end{equation*}
\end{theorem}

Based on this estimator, similar results as stated in Sections \ref{sec3} and \ref{sec4} can be derived. To be precise, define 
\[
d_{\infty, n}^{\sigma} = \textstyle\sup_{t \in I_n}|{d}^\sigma(t)|, 
\qquad
\hat d_{\infty, n}^{\sigma} = \textstyle\sup_{t \in I_n}|\hat{d}_n^\sigma(t)|, 
\]
where $\hat{d}_n^\sigma(t) = \tilde{\mu}_{h_n}(t)/\hat{\sigma}(t)-\hat g_n$.
Moreover, for some positive sequence $(\rho_n)_{n\in\N}=o(1)$, let ${\Ec}^\sigma_n = {\Ec}_n^+ \cup {\Ec}_n^-$  and $\hat{\Ec}^\sigma_n = \hat{\Ec}_n^+ \cup \hat{\Ec}_n^-$, where
\[ 
{\Ec}_n^\pm = \{ t\in I_n:  d_{\infty, n}^{\sigma}  \mp {d}_n^\sigma(t)\leq \rho_n \}, 
\qquad
\hat{\Ec}_n^\pm = \{ t\in I_n:  \hat d_{\infty, n}^{\sigma}  \mp \hat{d}_n^\sigma(t)\leq \rho_n \}. 
\]
Let $G^\sigma_{n,j}$ be defined as in \eqref{eq:Gn} ($j=1$) and \eqref{eq:Gn2} $(j=2)$, but with $\Ec_n$ replaced by $\Ec_n^\sigma$.

\begin{theorem}\label{thm:convLS} 
		Suppose that Assumptions \ref{assump:kernel}, \ref{assump:errorLS}, \ref{assump:ghatG} and \ref{assump:muLS} hold, and  assume that $h_n\to 0$, $\tau_n\to 0$ and $m_n\to\infty$, such that $nh_n\to\infty$ and 
		\begin{align*}
		&|\log(h_n)|\tfrac{\log^{4}n}{n^{1/2}h_n}\leq C<\infty,  \\
		&|\log(h_n)| n h_n^{7}\to 0,\\
		&|\log(h_n)|^{1/2} \Big(\tfrac{m_n^{1/4}\sqrt{h_n}}{\tau_n}+\tfrac{\sqrt{nh_n}}{m_n}+\sqrt{nh_n}\tau_n^2+\sqrt{\tfrac{h_nm_n^{5}}{n}}\Big)\to 0,
		\end{align*}
		Moreover, assume that $\rho_n\to 0$ and $\rho_n^2 nh_n / |\log (h_n)|\to\infty$, and if $\lambda(\Ec^\sigma)=0$, assume additionally $\rho_n^{\scs 1/2}h_n^{\scs-1}\to \infty$. 
		If $d_\infty^\sigma>0$, then 
\[  
\pr\Big(\tfrac{\ell_n(\hat \Ec_n^\sigma)\sqrt{nh_n}}{\|K^*\|_2} \big( \hat d_{\infty, n}^\sigma - d_{\infty, n}^\sigma\big)-\ell_n^2(\hat \Ec_n^\sigma)\leq x\Big)
\geq 
\pr(G_{n,1}^\sigma \leq x)+o(1), 
\]
with equality if $\lambda(\Ec^\sigma)>0$.  If $d_\infty^\sigma=0$, then
\[  
\pr\Big(\tfrac{\ell_n(\hat \Ec_n^\sigma)\sqrt{nh_n}}{\sigma\|K^*\|_2} \big( \hat d_{\infty, n} - d_{\infty, n}\big)-\ell_n^2(\hat \Ec_n^\sigma)\leq x\Big) 
= 
\pr(G_{n,2}^\sigma\leq x)+o(1).
\]	
Moreover,
$
		G_{n,1}^\sigma \convw \Gum_0$ and $G_{n,2}^\sigma \convw \Gum_{\log(2)}.
$
\end{theorem}

Tests for the hypothesis \eqref{h0_LS} can be derived in a similar way as in Section~\ref{sec4}. Exemplary, we consider the analogue of test \eqref{test:basedOnEstimator} (a test based on using the representation of $G^\sigma_{n,j}$ can be derived similarly). The null hypothesis in \eqref{h0_LS} is rejected whenever
\begin{equation} 
\label{test:basedOnEstimatorLS} 
\sup_{t\in I_n}|\hat{d}_n^\sigma(t)| > \{ q_{a,1-\alpha}+\ell_n^2(\hat\Ec_n^\sigma)\}\frac{\|K^*\|_2}{\sqrt{nh_n}\ell_n(\hat\Ec_n^\sigma)}+\Delta, 
\end{equation}
where $a=0$ if $\Delta>0$ and $a=\log(2)$ if $\Delta=0$.

\begin{corollary}  \label{test3}
	Under the assumptions of Theorem \ref{thm:convLS}, the test defined by the decision rule \eqref{test:basedOnEstimatorLS} is consistent and has asymptotic level $\alpha$. 
\end{corollary}

\section{Finite sample results}
\label{sec7}

We investigate the finite sample properties of the new methodology by means of a simulation study and illustrate its application in a data example.

\subsection{Monte Carlo simulation study} \label{sec71}

A large scale Monte Carlo simulation study was performed to analyse the finite-sample properties of the proposed tests. Two classes of mean functions $\mu$ were considered, the first one with $\lambda(\Ec)=0$ and the second one with $\lambda(\Ec)>0$.

The first class of models is based on the mean function
\[
\mu_a^{(1)}(x) = 10 + \tfrac{1}{2} \sin(8\pi x)+a\big(x-\tfrac{1}{4}\big)^2 \id\big(x>\tfrac{1}{4}\big),
\]
which is plotted in Figure \ref{fig:models} for various choices of $a$.  We considered the testing problem in \eqref{h0avl} with 
$x_0=1/4 $
 and with  $\Delta=1$, that is
\begin{equation} 
 \label{h0avlspec} 
H_0:     d_{\infty} =  \sup_{t \in [1/4,1] } \Big  | \mu_{a}^{(1)} (t) - 4\int_{0}^{1/4} \mu_{a}^{(1)} (s)  \diff s \Big |  \leq 1  \, \,  \, \,  \, \quad
\text{  vs.\ }  \quad H_1:  d_{\infty}   >  1.
\end{equation}
 Such a scenario might for instance be encountered and of interest in the context of climate change. Note that $\|\mu_{a^*}^{\scs (1)}-g(\mu_{a^*}^{\scs (1)})\|=\Delta$ for $a^*=\tfrac{128}{81}\approx 1.58$, whereas for $a<a^*$ and $a>a^*$ we have $\|\mu_{a}^{\scs (1)}-g(\mu_{a}^{\scs (1)})\|_\infty<\Delta$ and $\|\mu_{a}^{\scs (1)}-g(\mu_{a}^{\scs (1)})\|_\infty>\Delta$, respectively. 

The second class of models is based on the mean function
$$\mu^{(2)}(x) = \left\{\begin{array}{ll}
9&\quad \text{for}~x\le \tfrac{1}{4}\\
\tfrac{3}{2}\sin(2\pi x)+10.5&\quad \text{for}~\tfrac{1}{4}<x\le \tfrac{3}{4}\\
12&\quad \text{for}~\tfrac{3}{4}<x,
\end{array} \right. $$
again plotted in Figure \ref{fig:models}. For models involving $\mu^{(2)}$, we considered the testing problem in \eqref{h0all}
with 
$ x_0=0,$ $x_1=1$, $g(\mu)\equiv 10 $  and various choices of $\Delta > 0$, that is 
\begin{equation}
 \label{h0allspec} 
H_0:     d_{\infty} =  \sup_{t \in [0,1] } \Big  | \mu^{(2)} (t) - 10  \Big |  \leq \Delta  
\quad
\text{   vs.\  }   \quad H_1:  d_{\infty}   >   \Delta.
\end{equation}
 Such a setting might be encountered in quality control, where deviations from a target value might occur gradually due to wear and tear (and eventual failure) of a component of a complex system.
Note that  $\|\mu^{(2)}-g(\mu^{(2)})\|_\infty\le\Delta$ for $\Delta \ge 2$, whereas $\|\mu^{(2)}-g(\mu^{(2)})\|_\infty>\Delta$ for $\Delta < 2$.

For both choices of the mean function $\mu$ we chose  three different  error processes $(\eps_i)_{i\in\Z}$ in model \eqref{Lipschiss}, that is 
\begin{align*}
(\text{IID}) \quad &~ \eps_i= \tfrac{1}{2}\eta_i\\
(\text{MA)} \quad &~ \eps_i = \tfrac{1}{\sqrt{5}}\big( \eta_i + \tfrac{1}{2}\eta_{i-1}\big)\\
(\text{AR})  \quad &~ \eps_i = \tfrac{\sqrt{3}}{4} \big( \eta_i + \tfrac{1}{2}\eps_{i-1}\big),
\end{align*}
were $(\eta_i)_{i\in\Z}$ is an i.i.d.\ sequence of standard normally distributed random variables.
In particular, we have $\var(\eps_i)=\tfrac{1}{4}$ for all error processes under consideration.

\begin{figure}[tbp] 
	\centering \vspace{-.4cm}
	\mbox{\hspace{-.4cm}\includegraphics[width=69mm]{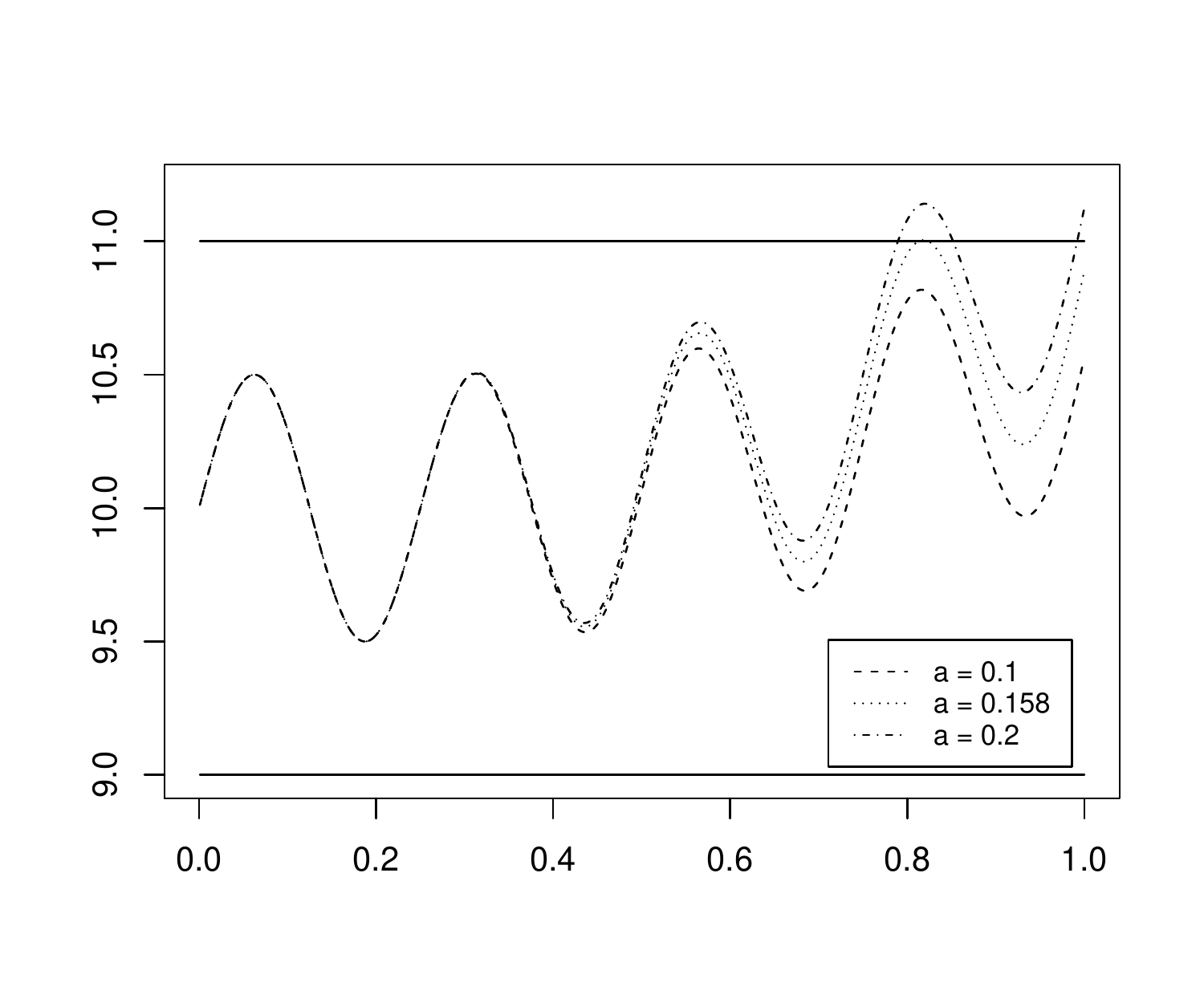} \hspace{-.5cm}
	\includegraphics[width=69mm]{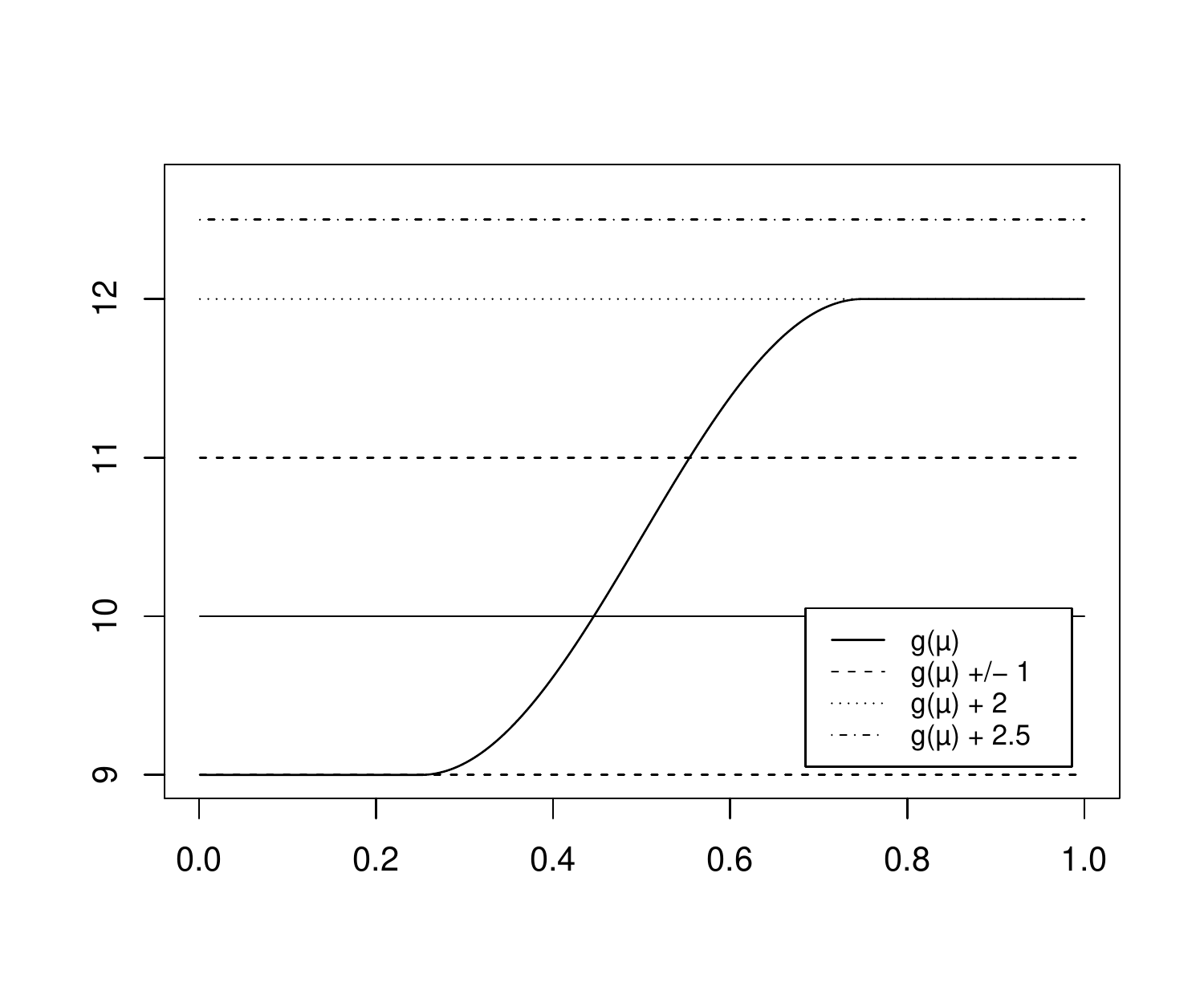}}\vspace{-.5cm}
	\caption{\textit{Left: The mean function $\mu_a^{\scs (1)}$ is plotted for three choices of $a$. Right: The mean function $\mu^{(2)}$ is plotted, alongside with $g+\Delta$ for five choices of the threshold $\Delta$.}} \vspace{-.2cm}
		\label{fig:models}
\end{figure}

The choice of the bandwidth $h_n$ for the estimator $\tilde{\mu}_{h_n}$ is crucial to avoid both overfitting and oversmoothing. For this purpose, we employ the following $k$-fold cross-validation procedure with $k=10$ (as recommended by \citealp{HasTibFri09}, page 242).

\begin{alg}[Cross-Validation for the Choice of $h_n$] \label{alg:hn}~
	\begin{compactenum}[1.]
		\item Split the observed data randomly in $k=10$ sets $S_1, \dots, S_{10}$ of equal length.
		
		\item For $h_n=\tfrac{1}{n}$ and each set $S_i$, calculate the Jackknife estimator $\tilde{\mu}_{h_n}^{(i)}$ based on the data in the remaining sets.
		
		\item Based on the Jackknife estimators $\tilde{\mu}_{h_n}^{(i)}$ from Step (2), compute the mean squared prediction error 
			\[
			\MSE_{h_n}=\frac{1}{1-h_n/2} \sum_{i=1}^{10}\sum_{j\in S_i} \big\{ X_{j,n}-\tilde{\mu}_{h_n}^{(i)}(j/n)\big\}^2.
			\]
		
		\item Repeat Steps (2) and (3) for the bandwidths $h_n=\tfrac{2}{n},\dots,\tfrac{\lfloor n/2 \rfloor}{n}$
		
		\item Choose the bandwidth $h_n$ that minimises the mean squared prediction error $\MSE_{h_n}$.	
	\end{compactenum}
\end{alg}

{\footnotesize
\begin{table}[t!]
			\begin{tabular}{l|r | rrr | rrr | rrr }		\hline \hline
		$\mu_a^{(1)}$&&\multicolumn{3}{c|}{test \eqref{test0} }&\multicolumn{3}{c|}{test \eqref{test1rel}} & \multicolumn{3}{c}{test \eqref{test:approxQuantiles}}\\
			$a$ &$d_\infty-\Delta$ & 200 & 500 & 1000 & 200 & 500 & 1000 & 200 & 500 & 1000 \\ 
		 \hline 
		\addlinespace[.2cm]
		\multicolumn{11}{l}{\quad\textit{Panel A: iid errors}} \\ 
		1.0 & -0.18 & 0.0 & 0.0 & 0.0 & 0.0 & 0.0 & 0.0 & 0.0 & 0.0 & 0.0 \\ 
		1.5 & -0.03& 0.0 & 0.0 & 0.0 & 0.0 & 0.0 & 0.0 & 0.0 & 0.0 & 0.0 \\ 
		\bf 1.58 & \bf 0.00 & \bf 0.0 & \bf 0.0 & \bf 0.0 & \bf 0.0 & \bf 0.0 & \bf 0.0 & \bf 0.0 & \bf 0.0 & \bf 0.2 \\ 
		2.0 & 0.13 & 0.0 & 0.1 & 1.1 & 0.0 & 0.2 & 2.4 & 0.0 & 3.3 & 23.1 \\ 
		2.5 & 0.29 & 0.0 & 5.6 & 66.9 & 0.0 & 9.0 & 77.6 & 0.0 & 29.9 & 97.8 \\ 
		3.0 & 0.45 & 0.0 & 34.8 & 99.8 & 0.0 & 39.9 & 99.9 & 0.2 & 57.3 & 100.0	 \\
	\addlinespace[.2cm]
		\multicolumn{11}{l}{\quad\textit{Panel B: MA errors}} \\ 
	1.0 & -0.18 & 0.0 & 0.0 & 0.0 & 0.0 & 0.0 & 0.0 & 0.0 & 0.0 & 0.0 \\ 
	1.5 & -0.03 & 0.0 & 0.0 & 0.0 & 0.0 & 0.0 & 0.0 & 0.0 & 0.0 & 0.0 \\ 
	\bf 1.58 & \bf 0.00 & \bf 0.0 & \bf 0.0 & \bf 0.0 & \bf 0.0 & \bf 0.0 & \bf 0.0 & \bf 0.0 & \bf 0.1 & \bf 0.3 \\ 
	2.0 & 0.13 & 0.0 & 0.3 & 0.9 & 0.0 & 0.5 & 2.1 & 0.0 & 3.7 & 18.7 \\ 
	2.5 & 0.29 & 0.1 & 4.6 & 40.2 & 0.1 & 7.4 & 51.7 & 0.2 & 27.0 & 87.9 \\ 
	3.0 & 0.45 & 0.1 & 25.4 & 96.1 & 0.1 & 31.0 & 97.9 & 0.5 & 52.8 & 99.7 \\ 
 \addlinespace[.2cm]
		\multicolumn{11}{l}{\quad\textit{Panel C: AR errors}} \\ 
	1.0 & -0.18 & 0.0 & 0.0 & 0.0 & 0.0 & 0.0 & 0.0 & 0.0 & 0.0 & 0.0 \\ 
	1.5 & -0.03 & 0.0 & 0.0 & 0.0 & 0.0 & 0.0 & 0.0 & 0.1 & 0.5 & 1.0 \\ 
	\bf 1.58 & \bf 0.00 & \bf 0.0 & \bf 0.1 & \bf 0.0 & \bf 0.0 & \bf 0.3 & \bf 0.0 & \bf 0.1 & \bf 1.4 & \bf 1.5 \\ 
	2.0 & 0.13 & 0.0 & 0.8 & 1.8 & 0.0 & 1.4 & 3.3 & 0.0 & 7.8 & 23.1 \\ 
	2.5 & 0.29 & 0.0 & 4.9 & 29.4 & 0.1 & 8.4 & 40.2 & 0.4 & 27.3 & 77.7 \\ 
	3.0 & 0.45 & 0.1 & 21.2 & 86.4 & 0.1 & 27.4 & 90.0 & 1.1 & 53.9 & 98.4 \\ \hline \hline
	\end{tabular}  \smallskip
\caption{\it Empirical rejection rates of various tests for the hypotheses  \eqref{h0avlspec},  different values for the parameter~$a$, different error processes, and sample sizes $n=200, 500, 1000$. }  \vspace{-.6cm}
\label{tab:mu1}
\end{table}
}

Throughout, 
we employed the quartic kernel $K(x)=\tfrac{15}{16}(1-x^2)^2$ for the local linear estimator. Preliminary simulation studies showed that different choices of the kernel led to similar results. The level $\alpha$ was chosen as $5\%$ and $\rho_n$ has been set to $\rho_n={\ell_n^{1+\eps}}/{\sqrt{nh_n}}$ with $\eps=0.001$. The block length of the long-run variance estimator was chosen as 
\[
m_n=\max\big\{\big\lfloor \sqrt{\tfrac{ | \hat \gamma_1|+\dots+|  \hat \gamma_4|}{| \hat \gamma_0|+\dots+| \hat \gamma_4|}}n^{1/3}\big\rfloor,1\big\},
\] 
where $ \hat \gamma_k$ denotes the empirical autocovariance  at lag $k$ of the residuals $\hat{\eps}_{i,n}=X_{i,n}-\tilde{\mu}_{h_n}(i/n)$, for $k=0,\dots,4$. Note that $m_n$ naturally adapts to the serial dependence of the residuals, with $m_n=1$  if the absolute  empirical autocorrelations are small.
The quantiles of $G_{n,j}$ are calculated  by  $2000$ simulation runs. The empirical rejection rates of the null hypothesis $H_0: d_\infty\leq\Delta$ are based on $N=1000$ simulation runs each and are displayed in Tables~\ref{tab:mu1} and \ref{tab:mu2} 
for the test  \eqref{test0} (based on the confidence band),  
the test \eqref{test1rel} (based on the estimate of the sup-norm and  the bound for the quantile of the limit distribution)
and  for the test \eqref{test:approxQuantiles} (based on the estimates of the sup-norm and the Lebesgue measure of the extremal sets).  Results for the  test \eqref{test:basedOnEstimator}, which is the analogue of \eqref{test:approxQuantiles} but with quantiles depending on the Gumbel-distribution, are not presented as they were always inferior to those of \eqref{test:approxQuantiles}.
 The sample size was chosen as  $n = 200,$ $500$  and $1000$. For the  tests   \eqref{test0} and  \eqref{test1rel}  we used
 $\ell_n'$ as defined in \eqref{eq:ln2} instead of $\ell_n$ in \eqref{eq:ln} (as pointed out in Section \ref{sec3} this makes asymptotically no difference).  The lines marked in boldface indicate the boundary of the null hypothesis, that is, the
 parameter where  $d_{\infty} = \Delta$. In order to achieve large power it is desirable that the empirical level of the test  is close to the nominal level $\alpha$  for those models.

{\footnotesize
	\begin{table}[t]
		\begin{tabular}{l|r | rrr | rrr | rrr }		\hline \hline
			$\mu^{(2)}$&&\multicolumn{3}{c|}{test \eqref{test0}}&\multicolumn{3}{c|}{test \eqref{test1rel}} & \multicolumn{3}{c}{test \eqref{test:approxQuantiles}}\\
			$\Delta$ &$d_\infty-\Delta$ & 200 & 500 & 1000 & 200 & 500 & 1000 & 200 & 500 & 1000 \\ 
			\hline
			 \addlinespace[.2cm]
		\multicolumn{11}{l}{\quad\textit{Panel A: iid errors}} \\ 
			1.0 & 1.0 & 99.9 & 100.0 & 100.0 & 100.0 & 100.0 & 100.0 & 100.0 & 100.0 & 100.0 \\ 
			1.5 & 0.5 & 70.1 & 99.0 & 100.0 & 76.0 & 99.5 & 100.0 & 92.4 & 99.9 & 100.0 \\ 
			1.75 & 0.25 & 2.8 & 18.5 & 94.2 & 5.9 & 28.1 & 96.9 & 43.3 & 73.6 & 99.7 \\ 
			\bf 2.0 & \bf 0.0 & \bf 0.0 & \bf 0.0 & \bf 0.0 & \bf 0.0 & \bf 0.0 & \bf 0.0 & \bf 0.0 & \bf 0.2 & \bf 3.0 \\ 
			2.25 & -0.25 & 0.0 & 0.0 & 0.0 & 0.0 & 0.0 & 0.0 & 0.0 & 0.0 & 0.0 \\
			 \addlinespace[.2cm]
		\multicolumn{11}{l}{\quad\textit{Panel B: MA errors}} \\ 
			1.0 & 1.0 & 97.9 & 100.0 & 100.0 & 98.8 & 100.0 & 100.0 & 100.0 & 100.0 & 100.0 \\ 
			1.5 & 0.5 & 46.9 & 93.9 & 100.0 & 58.2 & 97.2 & 100.0 & 86.0 & 99.7 & 100.0 \\ 
			1.75 & 0.25 & 2.6 & 11.0 & 76.3 & 5.0 & 16.9 & 83.9 & 32.5 & 61.3 & 97.8 \\ 
			\bf 2.0 & \bf 0.0 & \bf 0.0 & \bf 0.0 & \bf 0.1 & \bf 0.0 & \bf 0.0 & \bf 0.1 & \bf 0.6 & \bf 0.4 & \bf 3.7 \\ 
			2.25 & -0.25 & 0.0 & 0.0 & 0.0 & 0.0 & 0.0 & 0.0 & 0.0 & 0.0 & 0.0 \\
			 \addlinespace[.2cm]
		\multicolumn{11}{l}{\quad\textit{Panel C: AR errors}} \\ 
			1.0 & 1.0 & 96.6 & 100.0 & 100.0 & 98.0 & 100.0 & 100.0 & 99.8 & 100.0 & 100.0 \\ 
			1.5 & 0.5 & 42.7 & 84.8 & 99.8 & 51.7 & 90.1 & 100.0 & 80.1 & 99.2 & 100.0 \\ 
			1.75 & 0.25 & 5.6 & 11.1 & 55.9 & 8.7 & 17.3 & 67.1 & 33.4 & 56.1 & 91.5 \\ 
			\bf 2.0 & \bf 0.0 & \bf 0.0 & \bf 0.0 & \bf 0.1 & \bf 0.2 & \bf 0.2 & \bf 0.5 & \bf 2.1 & \bf 1.2 & \bf 5.1 \\ 
			2.25 & -0.25 & 0.0 & 0.0 & 0.0 & 0.0 & 0.0 & 0.0 & 0.0 & 0.0 & 0.0 \\
			\hline \hline
		\end{tabular} \smallskip
		\caption{\it 
		Empirical rejection rates of  different tests  for the hypotheses \eqref{h0allspec} for $\mu=\mu^{(2)}$, different error processes, different choices of $\Delta$ and sample sizes $n=200, 500, 1000$.}  \vspace{-.4cm}
		\label{tab:mu2}
	\end{table}
}	
	
To interpret the empirical rejection rates, note that the null hypothesis in the models involving  $\mu_a^{(1)}$ is true if and only if $a\leq a^* \approx 1.58$. Likewise, for the models involving $\mu^{(2)}$ the null hypothesis is true if and only if $\Delta\ge 2$. It can be seen that all tests under consideration are conservative, in particular for the models involving $\mu_a^{\scs (1)}$. Recall that the theory  in Sections \ref{sec2}--\ref{sec4} suggests that tests  \eqref{test0} and \eqref{test1rel} should be conservative for all models under consideration, while test  \eqref{test:approxQuantiles} should either be conservative, or yield rejection rates close to the  nominal level on the boundary of the null hypothesis for models involving $\mu^{(2)}$ (for which $\mu(\Ec)>0$). The empirical findings perfectly correspond to this theoretical prediction.
In terms of power, the results are similar, with test \eqref{test:approxQuantiles} clearly being the most powerful. The superiority in terms of power  is especially visible for alternatives for which $d_\infty-\Delta$ is small (say, around or below 1/4, which is half the standard deviation of the noise $\eps_t$). For $d_\infty-\Delta$ as low as 0.13 (model $\mu_a^{\scs (1)}$ with $a=2$), test \eqref{test:approxQuantiles} is the only test with a non-trivial power.

\begin{figure}[t!] 
\vspace{-.4cm}
	\mbox{\hspace{-.05cm}
	\includegraphics[width=69mm]{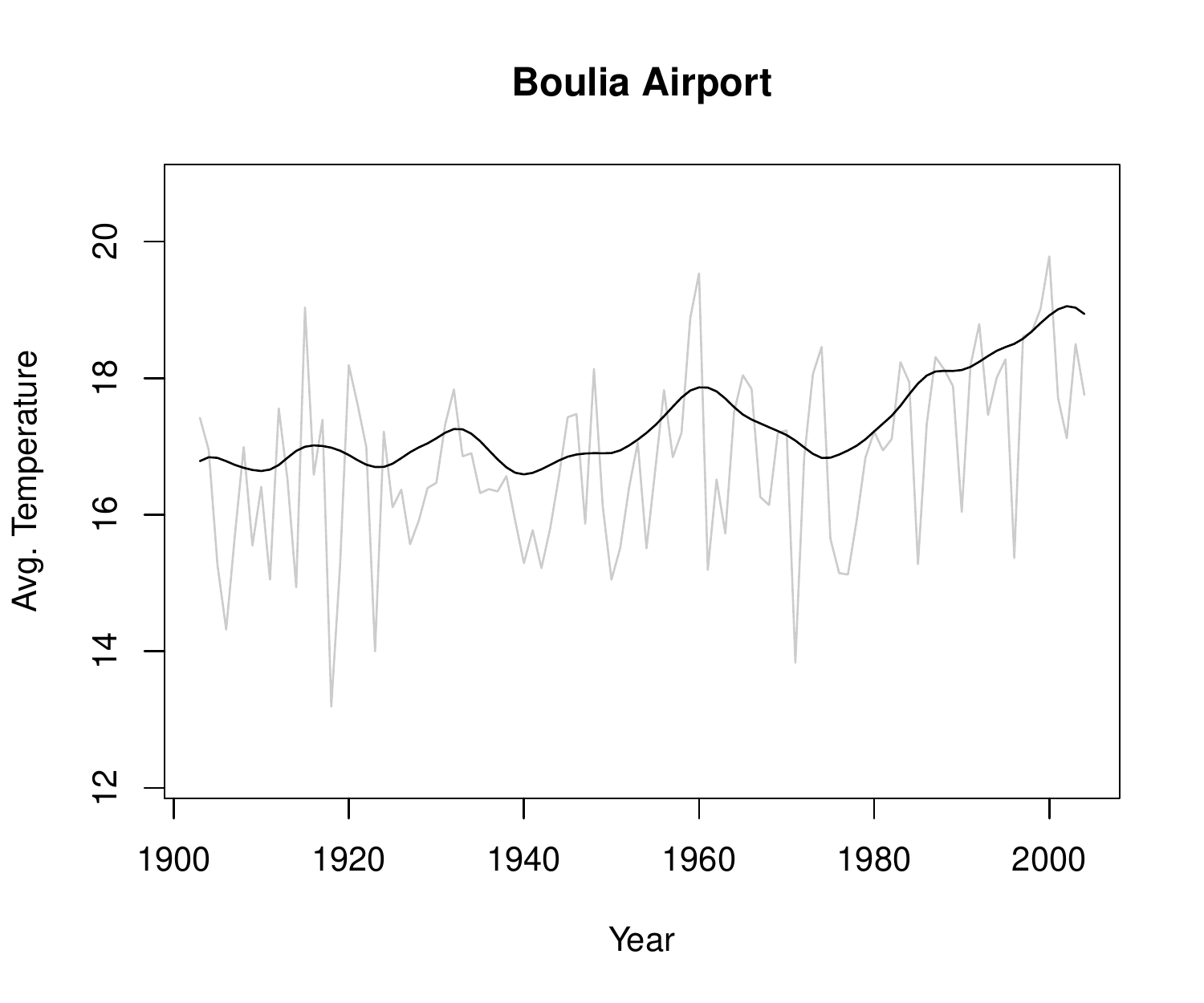}
	\includegraphics[width=69mm]{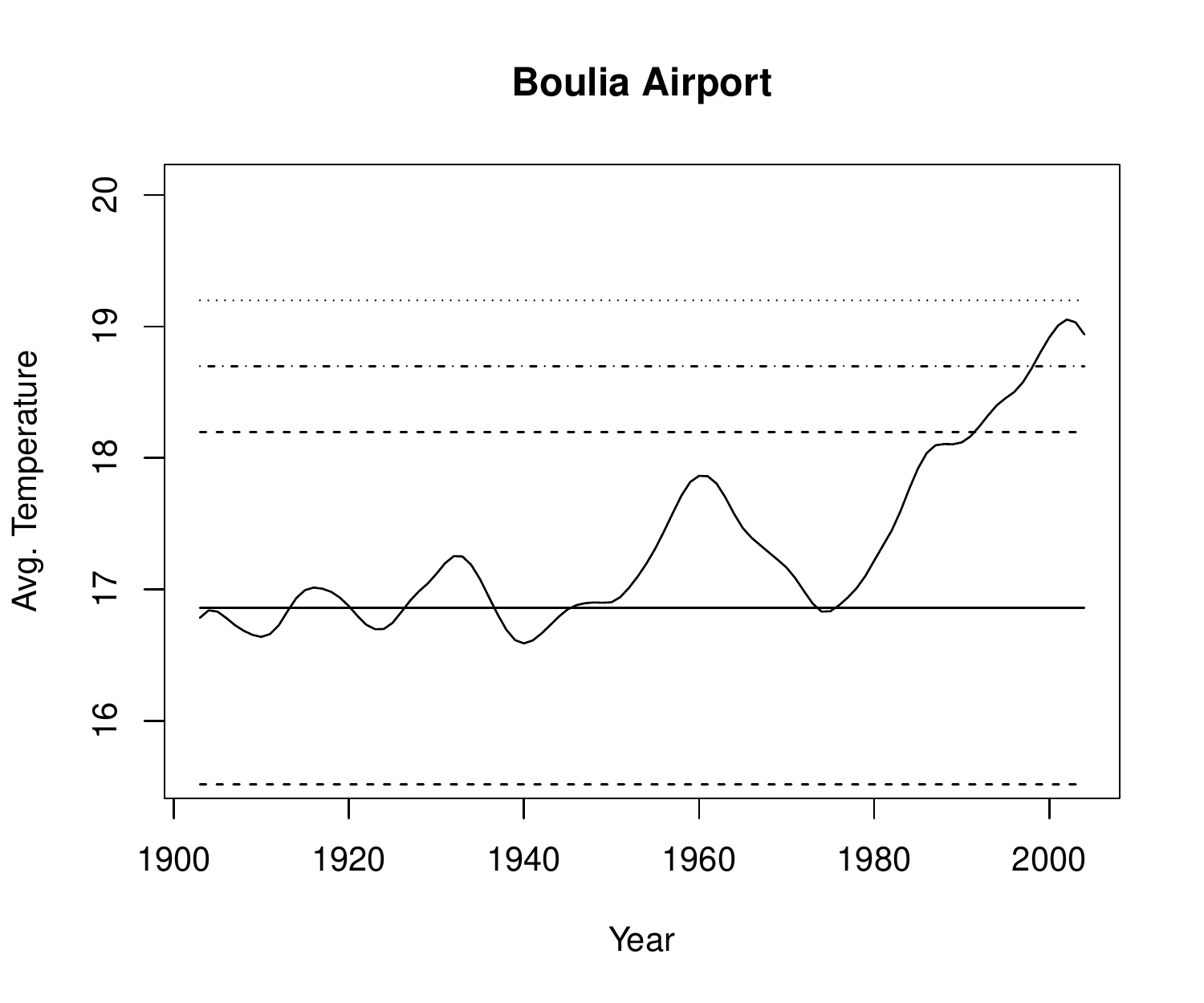}}  \\
	\vspace{-.28cm}
	\mbox{\hspace{-.05cm}
	\includegraphics[width=69mm]{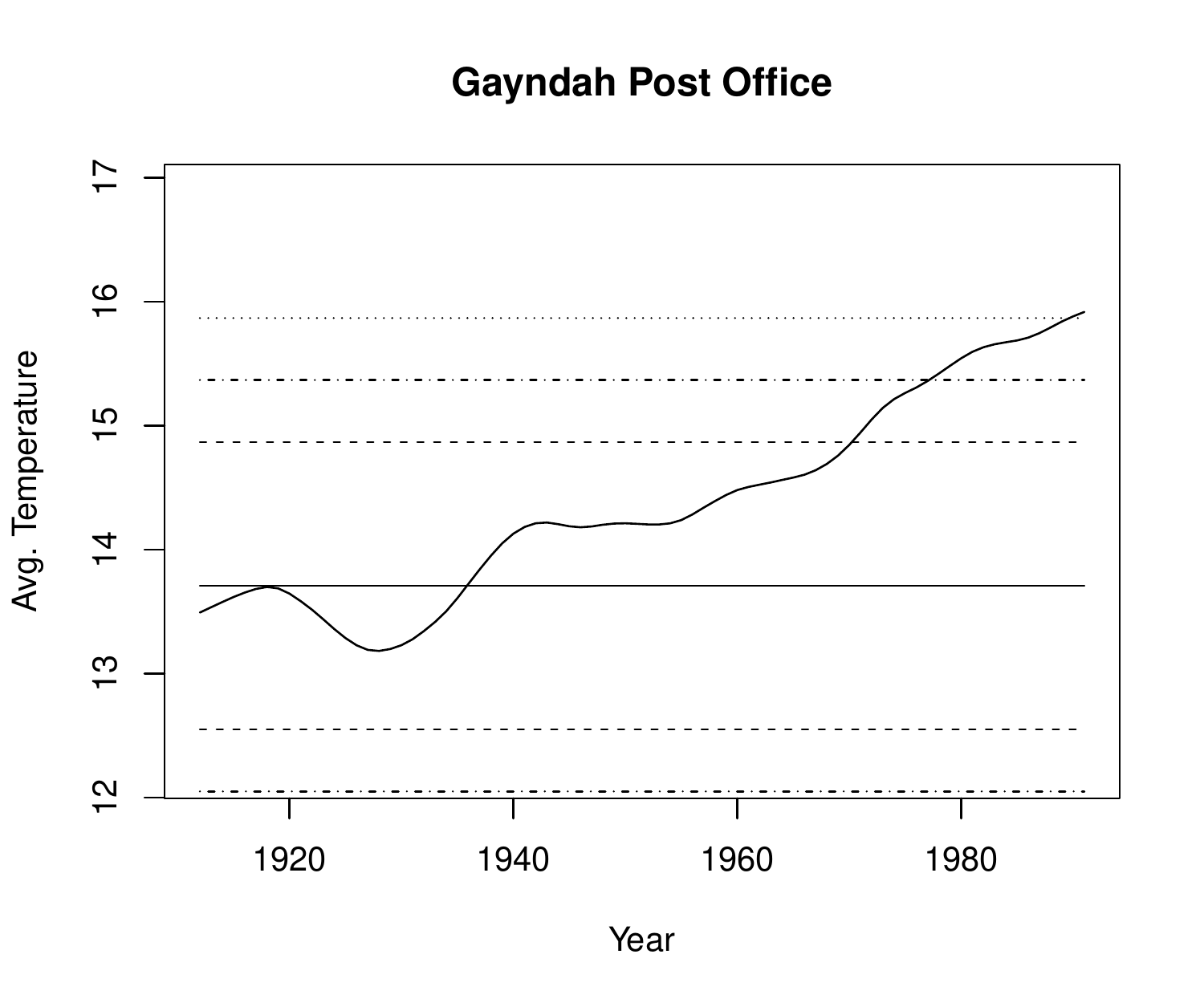}
	\includegraphics[width=69mm]{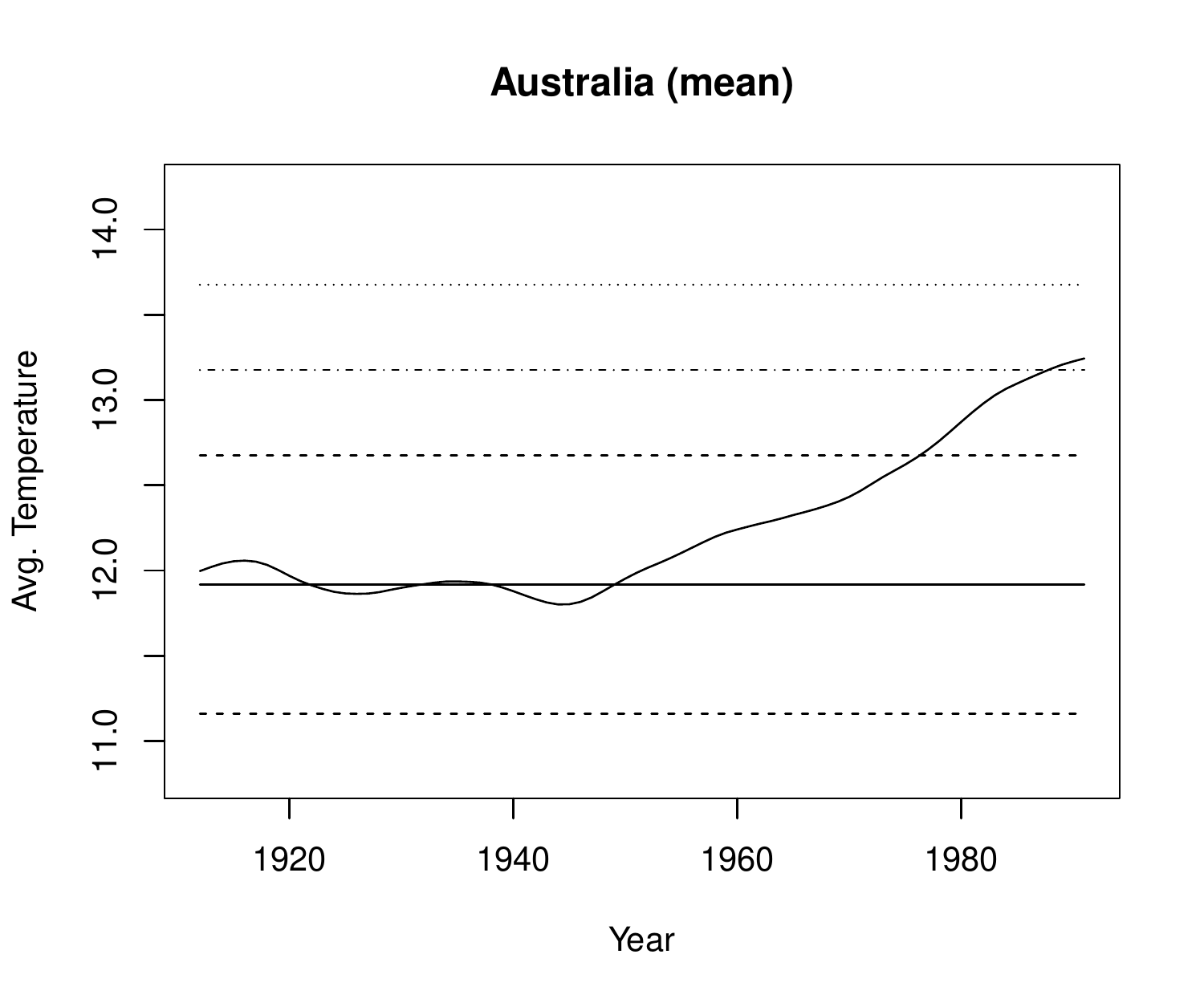}}
	\vspace{-.4cm}
	\caption{\it{ Top left: Raw data of the temperature in Boulia. Top Right: estimate $\tilde \mu_{h}$ for the mean temperature in Boulia (solid curve); mean over the temperature from the beginning of the records until 1950 (solid straight line); boundary of the test decision for $\Delta=0.5$, $\Delta=1.0$ and $\Delta=1.5$ (dashed line, dotted and dashed line, dotted line, respectively). Bottom left: Same as top right, but for Gayndah Post Office. Bottom right: Same as top right, but for the mean over the all stations under consideration. }}
	\label{fig:realData}
\end{figure}

\subsection{Case Study}
Time series with possibly non-constant mean naturally arise in the field of meteorology. To illustrate the proposed methodology, we consider the mean of daily minimal temperatures (in degrees Celsius) over the month of July for a period of approximately 120 years at eight different places in Australia. 
Exemplary, the observed temperature curve at the weather station in Boulia is plotted in the upper left of Figure~\ref{fig:realData}, alongside with its estimated smooth mean curve $\tilde \mu$.
At each station $j$, we tested for relevant deviations of the temperature within the second half of the 20th century from $g_j$,  the mean temperature over a historic reference period ranging from the late 19th century to 1950 at that station. As a threshold $\Delta$, we chose $0.5,1$ and $1.5$ degrees Celsius. The estimated mean curve, alongside with a line corresponding to the overall mean $g_j$ and three lines corresponding to the critical values for test \eqref{test:approxQuantiles}  can be found in Figure~\ref{fig:realData} (Stations Boulia and Gayndah Post Office, as well as the mean over all weather stations).

The results for all stations under consideration can be found in Table~\ref{tab:app}, where we also provide estimates for the first point in time exhibiting a relevant deviation from the historic period. The $p$-values are highly significant at all but $2$ stations for $\Delta=0.5$ degree Celsius, and at all but 3 stations for $\Delta=1$  degree Celsius. For $\Delta=1.5$  degree Celsius, only two stations exhibit $p$-values slightly below $0.05$. Finally, it is worthwhile to mention that tests \eqref{test0} and \eqref{test1rel} yield no significant $p$-values at all. Thus, the proposed test in \eqref{test:approxQuantiles} is clearly more adequate in the given context  for detecting relevant changes.

{\small
\begin{table}[t!]
	\begin{tabular}{r | r r r | r r r  } \hline \hline
		$\Delta$ & \multicolumn{1}{c}{0.5} & \multicolumn{1}{c}{1.0} & \multicolumn{1}{c|}{1.5} & \multicolumn{1}{c}{0.5} & \multicolumn{1}{c}{1.0} & \multicolumn{1}{c}{1.5} \\
		\hline
		Boulia & \textbf{0.0} & \textbf{0.6} & 9.6 & 1957 & 1960& $\infty$ \\
		Cape Otway & 62.8 & 100.0 & 100.0 & $\infty$ & $\infty$ & $\infty$ \\
		Gayndah & \textbf{0.0} & \textbf{0.0} & \textbf{3.5} & 1951 & 1969 & 1974 \\
		Gunnedah & \textbf{0.0} & \textbf{0.1} & \textbf{4.3} & 1951 & 1956 & 1976\\
		Hobart & \textbf{3.1} & 98.0 & 100.0 & 1975  & $\infty$ & $\infty$ \\
		Melbourne & \textbf{0.0} & \textbf{0.2} & 73.5 & 1968  & 1976& $\infty$ \\
		Robe & 78.3 & 100.0 & 100.0 & $\infty$ & $\infty$ & $\infty$ \\
		Sydney & \textbf{0.0} & 8.8 & 96.0 & 1978 & $\infty$ & $\infty$ \\
		Australia (mean) & \textbf{0.0} & \textbf{1.4} & 99.7 & 1970 & 1982& $\infty$  \\
		\hline \hline
	\end{tabular} \smallskip
	\caption{\it $p$-values of test \eqref{test:approxQuantiles} for the respective null hypotheses in percent (left part)  and estimated time of first relevant deviation (right part). 
	 Significant $p$-values (below 0.05) are in boldface. }\vspace{-.4cm}
	 	\label{tab:app}
\end{table}
}


\begin{appendix}
\section{Proofs} \label{sec8} 

\subsection{Preliminaries}   \label{subsec:prel1} 

In this section we present some general results for   Model~\eqref{modneu} with a locally stationary error process. In Remark \ref{stationary} we specialize the results to the stationary case, as they are needed for the proofs of the statements in Section \ref{sec3} and \ref{sec4}.

 Recall the definition of the local  long-run variance  estimator
 $\hat \sigma^{2} (t) $ in \eqref{eq:defLocLRVest} and let
\begin{equation*}
Z_{n}(t) = \frac{\sqrt{nh_n}}{\|K^*\|_{2}} \bigg(\frac{\tilde{\mu}_{h_n}(t)}{\hat{\sigma}(t)}-\frac{\mu(t)}{\sigma(t)}\bigg), \qquad t \in (0,1).
\end{equation*}
We are going to prove weak convergence results for the supremum of the random functions  $Z_{n}$ and $|Z_n|$ over  sets $\mathcal A \subset [0,1]$.
Note that similar results  are available for closely related processes
  in the case of independent data \citep[see, e.\,g.,][among others]{Johnston1982, Xia1998, Proksch2016},
 but the problem is less well investigated in the dependent case \citep[see, e.\,g.,][among others]{Hansen2008, Li2017, cao2018}.

 \begin{theorem} \label{a1}
 
 Let $\Ac$ denote a compact subset of the interval $[x_0,x_1]$ (or $(0,1)$ if $x_0=0,x_1=1$) with positive Lebesgue measure $\lambda(\Ac)>0$. Assume that $\Ac$
can be represented as a finite union of disjoint compact intervals, that is  $\Ac=\bigcup_{i=1}^m [x_{i,1},x_{i,2}]$, the case $x_{i,1}=x_{i,2}$ being allowed. If Assumptions
\ref{assump:kernel}, \ref{assump:errorLS}  and \ref{assump:muLS}(i) are satisfied and if $h_n\to 0$, $\tau_n\to 0$ and $m_n\to\infty$, such that $m_n/n\to0$, $nh_n\to\infty$ and
\begin{align}
\label{a11} &  |\log(h_n)| \tfrac{\log^{4}n}{n^{1/2}h_n}\leq C<\infty,  \\
\label{a12} & |\log(h_n)| nh_n^{7} \to 0,\\
\label{a13} &  |\log(h_n)|^{1/2} \Big(\tfrac{m_n^{1/4}\sqrt{h_n}}{\tau_n}+\tfrac{\sqrt{nh_n}}{m_n}+\sqrt{nh_n}\tau_n^2+\sqrt{\tfrac{h_nm_n^{5}}{n}}\Big)\to 0,
\end{align}
 then, with $\ell_n(\Ac)$ as defined in \eqref{lnA}, 
 \begin{align*}
&\Big\{\sup_{t\in\Ac}Z_{n}(t)-\ell_n(\Ac)\Big\} \ell_n(\Ac) \weak \Gum_0, \quad 
\Big\{ \sup_{t\in\Ac} |Z_{n}(t)|-\ell_n(\Ac)\Big \}  \ell_n(\Ac) \weak \Gum_{\log(2)}.
\end{align*}
\end{theorem}
\begin{proof}
First observe that 
\begin{equation} \label{hd1} 
\sup_{t \in I_n} |Z_{n}(t)-Z_{n,1}(t)| = \Oc_\pr\Big(\tfrac{m_n^{1/4}\sqrt{h_n}}{\tau_n}+\tfrac{\sqrt{nh_n}}{m_n}+\sqrt{nh_n}\tau_n^2+\sqrt{\tfrac{h_nm_n^{5}}{n}}\Big) ,
\end{equation} 
where
\begin{equation*} 
 Z_{n,1}(t) = \tfrac{\sqrt{nh_n}}{\sigma(t)\|K^*\|_{2}} \big\{ \tilde{\mu}_{h_n}(t)-\mu(t)\big \}.
 \end{equation*}
Indeed, by Theorem \ref{locstat} and part (3) of Assumption \ref{assump:errorLS},
\begin{align*}
\sup_{t\in[\gamma_n,1-\gamma_n]}\bigg|\frac{1}{\hat{\sigma}(t)}-\frac{1}{\sigma(t)}\bigg|
&= \sup_{t\in[\gamma_n,1-\gamma_n]} \frac{1}{\sigma(t)\hat{\sigma}(t)\{ \sigma(t) + \hat \sigma(t)\}} \big|\sigma^2(t) - \hat\sigma^2(t) \big|  \\
&= \Oc_\pr\Big(\tfrac{m_n^{1/4}}{\sqrt{n}\tau_n}+\tfrac{1}{m_n}+\tau_n^2+\tfrac{m_n^{5/2}}{n}\Big).
\end{align*}
Recalling the definition of the Jackknife estimator $ \tilde{\mu}_{h_n} $ in \eqref{ejackest} 
and the definition of the kernel $K^*$ in \eqref{eq:defassump:kernel} it follows from  Lemma C.2 of the Supplementary Material of \cite{dettewu2019}
that 
\begin{equation}  \label{eq:c2dw}
\sup_{t\in I_n} \Big| \tilde{\mu}_{h_n}(t)-\mu(t)-\frac{1}{nh_n}\sum_{i=1}^{n}\eps_{i,n} K^*_{h_n}(\tfrac{i}{n}-t) \Big|=\mathcal{O}(h_n^3+1/(nh_n)).
 \end{equation}
By Proposition 5 of \cite{ZhouZhou2013}, on a possibly richer probability space, there is a sequence $(V_i)_{i\in\N}$ of independent, standard normally distributed random variables such that 
\[ 
\max_{1\leq j \leq n} \bigg|\sum_{i=1}^{j}\eps_{i,n}-\sum_{i=1}^{j}\sigma\big(\tfrac{i}{n}\big)V_i\bigg|=o_\pr(n^{1/4}\log^2n). 
\]
Then, with the notation $S_j=\sum_{i=1}^{j-1} \{ \eps_{i,n}-\sigma\big(\tfrac{i}{n}\big)V_i \}$ for $j\ge 1$, $S_0 = 0$ and $t\in I_n=[x_0 \vee h_n , x_1 \wedge (1-h_n)]$, summation by parts leads to
\begin{align*} 
&\phantom{{}={}} \bigg| \sum_{i=1}^{n} \eps_{i,n}K^*_{h_n}\big(\tfrac{i}{n}-t\big)-\sum_{i=1}^{n} \sigma\big(\tfrac{i}{n}\big)V_iK^*_{h_n}\big(\tfrac{i}{n}-t\big) \bigg|\\
&= \bigg| \sum_{i=1}^{n} \big(S_{i+1} - S_{i} \big)K^*_{h_n}\big(\tfrac{i}{n}-t\big) \bigg|\\
&= \bigg|K^*_{h_n}(1-t)S_{n+1}-K^*_{h_n}\big(\tfrac{1}{n}-t\big)S_1 - \sum_{i=2}^{n} S_i \Big(K^*_{h_n}\big(\tfrac{i}{n}-t\big)- K^*_{h_n}\big(\tfrac{i-1}{n}-t\big)\Big) \bigg|\\
&\leq \sum_{i=2}^{n}\big|K^*_{h_n}\big(\tfrac{i}{n}-t\big)- K^*_{h_n}\big(\tfrac{i-1}{n}-t\big)\big|o_\pr(n^{1/4}\log^2n)\\
&=
\sum_{i=\lceil nt-nh_n\rceil}^{\lceil nt+nh_n\rceil}\big|K^*_{h_n}\big(\tfrac{i}{n}-t\big) -K^*_{h_n}\big(\tfrac{i-1}{n}-t\big)\big|o_\pr(n^{1/4}\log^2n)= o_\pr(n^{1/4}\log^2n),
\end{align*}
because $K^*$ is Lipschitz continuous and $\supp(K^*)=[-1,1]$. Thus, it follows from \eqref{eq:c2dw}
\begin{equation} \label{eq:jackKnife}
\sup_{t\in I_n} \Big| \tilde{\mu}_{h_n}(t)-\mu(t)-\frac{1}{nh_n}\sum_{i=1}^{n}\sigma\big(\tfrac{i}{n}\big)V_i K^*_{h_n}(\tfrac{i}{n}-t) \Big|=\mathcal{O}(h_n^3)+o_\pr\big(\tfrac{\log^2n}{n^{3/4}h_n}\big). 
\end{equation}
and consequently
\begin{align} 
\sup_{t\in I_n} | Z_{n,1}(t) -  Z_{n,2}(t)| = \mathcal{O}(\sqrt{n}h_n^{7/2})+o_\pr\big(\tfrac{\log^2n}{n^{1/4}h_n^{1/2}}\big) ,  \label {hd2a} 
\end{align}
where
\begin{align*}
Z_{n,2}(t) &= \frac{1}{\sigma(t)\|K^*\|_{2}\sqrt{nh_n}}\sum_{i=1}^{n}\sigma\big(\tfrac{i}{n}\big)V_iK_{h_n}^*\big(\tfrac{i}{n}-t\big).
\end{align*}

Further, let
\begin{align}
Z_{n,3}(t)=\frac{1}{\|K^*\|_{2}\sqrt{nh_n}}\sum_{i=1}^{n}V_iK_{h_n}^*\big(\tfrac{i}{n}-t\big).
\label{hd30}
\end{align}
Recalling the definition of $\Lambda_{K}$ in \eqref{lamK} and $\ell_n(\Ac)$ in \eqref{lnA}  we obtain by similar arguments as in the proof of Theorem 2 in  \cite{Proksch2016}  that
\begin{align*} 
&\Big\{ \sup_{t\in\Ac} Z_{n,3}(t)-\ell_n(\Ac)\Big\} \ell_n(\Ac)\weak \Gum_{0}, 
\quad 
\Big\{ \sup_{t\in\Ac}| Z_{n,3}(t)|-\ell_n(\Ac)\Big\} \ell_n(\Ac) \weak \Gum_{\log(2)}.
\end{align*}
The assertion follows from the stochastic expansions \eqref{hd1}, \eqref{hd2a} and 
\begin{align} \label{hd3a}
\sup_{t\in I_n} |Z_{n,2}(t) - Z_{n,3}(t) | = o_\pr(1),
\end{align}
observing the assumptions \eqref{a11}, \eqref{a12} and \eqref{a13}.

It remains to prove \eqref{hd3a}. First observe that
\begin{align}\label{eq:axel1}
&\phantom{{}={}}\ex\big[\big(Z_{n,2}(t)-Z_{n,3}(t)-Z_{n,2}(s)+Z_{n,3}(s)\big)^2 \big]\\
&= \frac{1}{\|K^*\|_{2}^2nh_n}\sum_{i=1}^{n}\bigg\{\frac{\sigma(\tfrac{i}n)-\sigma(t)}{\sigma(t)}K_{h_n}^*\big(\tfrac{i}{n}-t\big)-\frac{\sigma(\tfrac{i}n)-\sigma(s)}{\sigma(s)}K_{h_n}^*\big(\tfrac{i}{n}-s\big)\bigg\}^2. \nonumber
\end{align}
Note that the only non zero summands are those with $|i/n-t|\leq h_n$ or $|i/n-s|\leq h_n$. In the following, we only consider the case $|i/n-s|\leq h_n$ as the case $|i/n-t|\leq h_n$ can be investigated with the same arguments. By Assumption \ref{assump:muLS}(i) and the mean value theorem it exists $\xi\in \big(\tfrac{i}{n}\wedge s,\tfrac{i}{n}\vee s\big)$ such that
\begin{align*}
&\phantom{{}={}}\frac{\sigma(i/n)-\sigma(t)}{\sigma(t)}K_{h_n}^*\big(\tfrac{i}{n}-t\big)-\frac{\sigma(i/n)-\sigma(s)}{\sigma(s)}K_{h_n}^*\big(\tfrac{i}{n}-s\big)\\
&=\frac{\sigma(i/n)[\sigma(s)-\sigma(t)]}{\sigma(s)\sigma(t)}K_{h_n}^*\big(\tfrac{i}{n}-t\big) + \frac{\sigma(i/n)-\sigma(s)}{\sigma(s)}\Big[K_{h_n}^*\big(\tfrac{i}{n}-t\big)-K_{h_n}^*\big(\tfrac{i}{n}-s\big)\Big]\\
&= \frac{\sigma(i/n)[\sigma(s)-\sigma(t)]}{\sigma(s)\sigma(t)}K_{h_n}^*\big(\tfrac{i}{n}-t\big) + \frac{\sigma'(\xi)}{\sigma(s)}\big(\tfrac{i}{n}-s\big)\Big[K_{h_n}^*\big(\tfrac{i}{n}-t\big)-K_{h_n}^*\big(\tfrac{i}{n}-s\big)\Big].
\end{align*}
Since $\sigma$ is bounded away from zero and Lipschitz continuous, the first summand on the right-hand side of the previous display can be bounded by $C|s-t|$. Since $|i/n-s|\leq h_n$ and $K^*$ is Lipschitz continuous, the second summand can be bounded by $C|s-t|$ as well. Hence, the non-zero summands in \eqref{eq:axel1} can be bounded by $C|s-t|^2$. As there are only $\Oc(nh_n)$ non-zero summands, the right-hand side of \eqref{eq:axel1} can be bounded by $C|s-t|^2$. With this bound, Theorem 2.2.4 of \cite{VanWel96} leads with $\psi(x)=x^2,\delta=2h_n$ and $\eta=\sqrt{h_n}$ to
$$ 
\ex\bigg[ \sup_{|s-t|<2h_n} \big\{ Z_{n,2}(t)-Z_{n,3}(t)-Z_{n,2}(s)+Z_{n,3}(s)\big\}^2 \bigg] \leq Ch_n^{1/2}. 
$$
With the notation 
$$
\Tc_n = \{2j nh_n+1:j\in\N \}\cap [nx_0\vee nh_n, nx_1\wedge (n-nh_n)],
$$ 
it follows by the triangle inequality
\begin{align}\label{eq:bound_sup_moment}
&\phantom{{}={}} 
\ex\Big[ \sup_{t\in I_n} \big\{ Z_{n,2}(t)-Z_{n,3}(t)\big\}^2 \Big]^{1/2}\\
&= \nonumber
\ex\bigg[ \max_{j\in \Tc_n} \sup_{|t-j/n|\leq 2h_n} \big\{ Z_{n,2}(t)-Z_{n,3}(t)\big\}^2 \bigg]^{1/2}\\
&\leq \nonumber
\ex\Big[ \max_{j\in \Tc_n} \big\{Z_{n,2}(\tfrac{j}{n})-Z_{n,3}(\tfrac{j}{n})\big\}^2 \Big]^{1/2}  \\
&\hspace{1.5cm} + \nonumber
\ex\bigg[ \sup_{|s-t|<2h_n} \big\{Z_{n,2}(t)-Z_{n,3}(t)-Z_{n,2}(s)+Z_{n,3}(s)\big\}^2 \bigg]^{1/2}\\
&\leq  
\ex\Big[ \max_{j\in \Tc_n} \big\{ Z_{n,2}(\tfrac{j}{n})-Z_{n,3}(\tfrac{j}{n})\big\}^2 \Big]^{1/2} + C\sqrt{h_n}. \nonumber
\end{align}
Observe that the indices in the maximum on the right-hand side of the previous display have a distance of $2nh_n$. Further, the summation in the definition of $Z_{n,2}(j/n)-Z_{n,3}(j/n)$ ranges from $\lfloor j-nh_n\rfloor$ to $\lfloor j+nh_n \rfloor$, thus, the random variables in the latter maximum are independent and, by definition, normally distributed. Further, observe that $|\Tc_n|\leq \tfrac{1}{2h_n}$. With the notation $Z_j= Z_{n,2}(j/n)-Z_{n,3}(j/n)$ and $\sigma_j^2=\var(Z_j)$, it follows that, for all $t\in(0, 1/ \{2 \max_{j \in \Tc_n} \sigma_j^2\})$,
\begin{align}\label{eq:bound_max_moment}
\exp\Big(t\ex\big[\max_{j \in \Tc_n} Z_j^2\big]\Big)
&\leq 
\ex\Big[\exp\big(t \max_{j \in \Tc_n} Z_j^2\big) \Big] 
=  
\ex\big[\max_{j \in \Tc_n} \exp(t  Z_j^2) \big]\\
&\leq \nonumber
\sum_{j \in \Tc_n} \ex\big[\exp(t  Z_j^2) \big] 
= \nonumber
\sum_{j \in \Tc_n} \ex\big[\exp(t \sigma_j^2 (Z_j/\sigma_j)^2) \big] \\
&= 
\sum_{j \in \Tc_n} \frac{1}{\sqrt{1-2t\sigma_j^2}}
\le  \frac{|\Tc_n|}{\sqrt{1-2t \max_{j \in \Tc_n} \sigma_j^2}},  \nonumber
\end{align}
where we have used the fact that the moment-generating function of the $\chi_1^2$ distribution is $(1-2t)^{-1/2}$. Now observe that, by Lipschitz-continuity of $\sigma$,  for all $j\in\Tc_n$,
\begin{align*}
\sigma_j^2&=\ex\big[ \big\{Z_{n,2}(j/n)-Z_{n,3}(j/n)\big\}^2 \big]\\
&= 
\frac{1}{\sigma^2(j/n)\|K^*\|_{2}^2nh_n}\sum_{i=1}^{n}\Big\{\sigma\big(\tfrac{i}{n}\big)-\sigma\big(\tfrac{j}{n}\big)\Big\}^2\Big\{K_{h_n}^*\big(\tfrac{i}{n}-\tfrac{j}{n}\big)\Big\}^2\leq Ch_n^2.
\end{align*}
Thus, $t_n:=(2Ch_n^2|\log(h_n)|)^{-1} \le  |\log h_n|^{-1}  \times \{2 \max_{j\in \Tc_n} \sigma_j^2\}^{-1}$ is a valid choice for $t$ in \eqref{eq:bound_max_moment}, which leads to
\begin{align*}
\ex\big[\max_{j=1}^{|\Tc_n|} Z_j^2\big] 
&\leq 
\frac{1}{t_n}\log\Big(\frac{|\Tc_n|}{\sqrt{1-2t_nCh_n^2}}\Big)\\
&= 
2Ch_n^2 |\log(h_n)| \log\bigg(\frac{1}{2h_n\sqrt{1-|\log(h_n)|^{-1}}}\bigg) =\Oc\big(h_n^2|\log(h_n)|^2 \big).
\end{align*}
Plugging this bound into \eqref{eq:bound_sup_moment} yields $\ex\big[ \sup_{t\in I_n} \big\{ Z_{n,2}(t)-Z_{n,3}(t)\big\}^2 \big]=\Oc(h_n)$. In particular, $\sup_{t\in I_n} |Z_{n,2}(t) - Z_{n,3}(t) | = \Oc_\pr(\sqrt{h_n})$, thus, \eqref{hd3a} follows.
\end{proof}

\begin{remark} \label{rem:convProksch} 
 Conditions \eqref{a11}--\eqref{a13} are for instance satisfied if $h_n=n^{-c}$, for some $c \in (3/7,1/2)$, $m_n=|\log h_n|^{1/2} \log n \sqrt{nh_n}$ and $\tau_n=m_n^{-1/2}$.
\end{remark}

\begin{remark} \label{stationary}
~
\begin{compactenum}[(i)]
\item	The representation of $\Ac$ as a union of finitely many intervals is needed to ensure the \textit{blowing up} property of $h_n^{-1}\Ac$ in order to apply Theorem 14.1 of \cite{piterbarg2012} in the proof of Theorem 2 of \cite{Proksch2016}. In fact, $\Ac$ can be replaced by any sequence $(\Ac_n)_{n\in\N}$ of subsets of $I_n$ with $(h_n^{-1}\Ac_n)_{n\in\N}$ satisfying the blowing up property (cf. Definition 14.1 in \citealp{piterbarg2012}).
	\item 
It follows from the proof of Theorem \ref{a1} that in the situation of a stationary error process as considered in Section \ref{sec2} - \ref{sec4} 
the weak convergence 
 \begin{align}\label{eq:convProkschStat}
\Big\{\sup_{t\in\Ac}Z_{n,j}(t)-\ell_n(\Ac)\Big\}\ell_n(\Ac) 
&\weak   \Gum_{0}, \\
\Big\{\sup_{t\in\Ac} |Z_{n,j}(t)|-\ell_n(\Ac)\Big\}\ell_n(\Ac) 
&\weak \Gum_{\log(2)},
\label{eq:convProkschAbsStat}
\end{align}
remains valid for $Z_{n,3}$ as defined in \eqref{hd30}, as well as for the processes
\begin{align*}
 Z_{n,1}(t) &= \frac{\sqrt{nh_n}}{\sigma \|K^*\|_{2}} \big \{ \tilde{\mu}_{h_n}(t)- \mu(t) \big \} ,\quad 
\hat Z_{n,1}(t) = \frac{\sqrt{nh_n}}{\hat \sigma \|K^*\|_{2}} \big \{ \tilde{\mu}_{h_n}(t)- \mu(t) \big \},
\end{align*}
where $\sigma^{2} $ is the long-run variance defined in \eqref{eq:defLRV}, with corresponding estimator $\hat \sigma^2$ defined in \eqref{lrvest}. Moreover
\begin{align}
\label{eq:zn1zn3}
Z_{n,1}=Z_{n,3}+o_\pr( |\log h_n|^{-1/2}). 
\end{align}
The proof does not require condition \eqref{a13} (as there is no varying long-run variance which has to be estimated). 
\end{compactenum}
\end{remark}

\begin{remark} \label{confband}
The assertion regarding $\hat Z_{n,1}$ in Remark~\ref{stationary}(ii) allows for the construction of simultaneous $(1-\alpha)$-confidence bands for the regression function $\mu$. 
	More precisely, a careful inspection reveals that we may replace $\ell_n(I_n)$ by $\ell_n=\ell_n([x_0, x_1])$ in the weak convergence result, which  implies that the collection of intervals
	\[ 
	\tilde I_n(t) = [ \tilde{\mu}_{h_n}(t)-c_{n,\alpha}, \tilde{\mu}_{h_n}(t)+c_{n,\alpha}], \qquad {t\in I_n},
	\]
	with
	\begin{equation} \label{cnalpha} 
	c_{n,\alpha} = \Big(q_{\log(2), 1-\alpha} +\ell_n^2\Big) \frac{\hat\sigma\|K^*\|_2}{\sqrt{nh_n}\ell_n},
	 \end{equation}  
	defines an asymptotic simultaneous $(1-\alpha)$-confidence band for  $\mu$  in model \eqref{Lipschiss}, i.e., 
		\begin{align*}
	\quad\lim_{n\to\infty} \Prob\big( \mu(t) \in \tilde I_n(t)  \ \forall\, t \in I_n\big) =
	\lim_{n\to\infty} \Prob( \sup_{t\in I_n} | \tilde \mu_{h_n}(t) - \mu(t) | \le c_{n,\alpha}) 
	= 1-\alpha.
	\end{align*}	
As a further consequence, if additionally Assumption \ref{assump:ghat}  is met, the intervals in \eqref{cb} 
define  an asymptotic simultaneous $(1-\alpha)$-confidence band for the function $\mu - g(\mu)$ and the 
	decision rule \eqref{test0} 
	is  a consistent, asymptotic level $\alpha$ test.
\end{remark}

\subsection{Proofs for Section  \ref{sec3}}   \label{subsec:proof3} 
\begin{proof}[Proof of Theorem \ref{thm:Dn}]   
The assertion in (3) is a consequence of Theorem~\ref{conv:DnEn} proven below: first, choose a positive sequence $\rho_n=o(1)$ such that 
$\rho_nh_n^{-2}\to\infty$.  With $\Ec_n$ and $\ell_n(\Ec_n)$ as defined in \eqref{eq:defE} and \eqref{lnA} respectively, it holds $\ell_n(\Ec_n) \to \infty$ (Proposition~\ref{prop:BlowingUp}) and $\ell_n(\Ec_n)=o(\ell_n)$ since $\lambda(\Ec_n)\to 0$. Further, it holds $d_\infty>0$. Thus, applying Theorem~\ref{conv:DnEn},
\begin{align*}
\tfrac{\ell_n \sqrt{nh_n}}{\sigma\|K^*\|_2} ( \hat d_{\infty,n} - d_{\infty, n} ) - \ell_n^2
&\ll_S
\ell_n\Big  ((G_{n,2} +o_\Prob(1)) / \ell_n(\Ec_n) + \ell_n(\Ec_n)-{\ell_n}\Big) ,
\end{align*}
which converges to $-\infty$ in probability as asserted.

Regarding (1) and (2),
recall the definition of $\ell_n$  in \eqref{eq:ln} and let
	\begin{equation}\label{dn} 
	D_n= \tfrac{\ell_n \sqrt{nh_n}}{\sigma\|K^*\|_2}\Big\{ \sup_{t\in I_n}|\tilde{\mu}_{h_n}(t)-g(\mu)|-\sup_{t\in I_n}|\mu(t)-g(\mu)|\Big\} -\ell_n^2.
	\end{equation}
	Observing that
	\[
	\tfrac{\ell_n \sqrt{nh_n}}{\sigma\|K^*\|_2}\Big\{ \sup_{t\in I_n}|\tilde{\mu}_{h_n}(t)-\hat g_n|-\sup_{t\in I_n}|\tilde{\mu}_{h_n}(t)-g(\mu)|\Big\} = o_\pr(1)
	\]
	  by Assumption \ref{assump:ghat}, the assertion of Theorem \ref{thm:Dn} is a consequence of the weak convergence
$
D_n \convw \Gum_{\log\{ \lambda(\Ec)/(x_1-x_0) \} }
$ 
if $d_\infty>0$, and $D_n\convw \Gum_{\log(2)}$ if $d_\infty=0$.

The statement in (2), i.e., for $d_\infty=0$, now follows from  Remark \ref{stationary}(i) and (ii).
The  statement in (1), i.e., for $d_\infty>0$, is a consequence of the  following  two propositions. 
\end{proof}

	\begin{proposition}\label{thm:DnE}
If  the assumptions of Theorem \ref{thm:Dn}(1) are satisfied, then
		\[ 
		D_n(\Ec)= \tfrac{\ell_n \sqrt{nh_n}}{\sigma\|K^*\|_2} \Big\{ \sup_{t\in\Ec\cap I_n}|\tilde{\mu}_{h_n}(t)-g(\mu)|-
		\sup_{t\in \Ec \cap I_n}|\mu(t)-g(\mu)|\Big\}-\ell_n^2 
		\]
		converges weakly to $\Gum_{\log\{ \lambda(\Ec)/(x_1-x_0) \} }$.	\end{proposition}

\begin{proposition}\label{thm:Rn}
If  the assumptions of Theorem \ref{thm:Dn}(1) are satisfied, then, with $D_n$ as defined in \eqref{dn},
\[
\Rc_n=D_n-D_n(\Ec)=o_\Prob(1), \qquad n \to\infty.
\]
\end{proposition}

	\begin{proof}[Proof of Proposition \ref{thm:DnE}]
		We prove the proposition for the case that both sets $\Ec^+$ and $\Ec^-$ are non-empty. The other cases follow by the same arguments. Throughout the proof, since $\lambda(\Ec)>0$, we may assume that $n$ is sufficiently large such that $d_{\infty,n} = d_\infty$ and $\sup_{t\in \Ec \cap I_n} |d(t)| = \sup_{t \in \Ec} |d(t)| =d_\infty$.
				
		Now, observe that 
		\begin{equation}\label{eq:ratioLnLnE}
		\frac{\ell_n}{\ell_n(\Ec)}
		= \sqrt{\frac{\log(\frac{(x_1-x_0)\Lambda_{K}}{2\pi h_n})}{\log(\frac{\lambda(\Ec)\Lambda_{K}}{2\pi h_n})}}
		= \sqrt{\frac{\log(x_1-x_0)+\log(\frac{\Lambda_{K}}{2\pi h_n})}{\log(\lambda(\Ec))+\log(\frac{\Lambda_{K}}{2\pi h_n})}}\to 1,
		\end{equation}
		as $\lambda(\Ec)>0$ by assumption.
		Recall the definition of $Z_{n,1}$ and $Z_{n,3}$ in Remark~\ref{stationary}(ii). Since $\mu(t)-g(\mu)=d_\infty$ for $t\in\Ec^+ \cap I_n$, we obtain from \eqref{eq:zn1zn3}
		\begin{align}\label{eq:muHatPos}
		& \phantom{{}={}} 
		\tfrac{\ell_n\sqrt{nh_n}}{\sigma\|K^*\|_2} \sup_{t\in\Ec^+\cap I_n} \Big \{ -\big(\tilde{\mu}_{h_n}(t)-g(\mu)\big)\Big \} \\
		\nonumber
		&=\ell_n\sup_{t\in\Ec^+\cap I_n}\Big\{  - Z_{n,1}(t) -  \tfrac{\sqrt{nh_n}}{\sigma\|K^*\|_2} \big(\mu(t)-g(\mu)\big)\Big\} \\ 
		\nonumber
		&=\ell_n\sup_{t\in\Ec^+\cap I_n}\Big\{  - Z_{n,3}(t) -  \tfrac{\sqrt{nh_n}}{\sigma\|K^*\|_2} d_\infty \Big\}+o_\pr(1)\\ 
		\nonumber 
		&\stackrel{\Dc}{=} \ell_n\sup_{t\in\Ec^+\cap I_n} Z_{n,3}(t) -\tfrac{\ell_n\sqrt{nh_n}}{\sigma\|K^*\|_2} d_\infty+o_\pr(1)\\
			\nonumber 
		&\leq \ell_n \Big(\sup_{t\in I_n} Z_{n,3}(t)-\ell_n\Big) -\ell_n\Big(\tfrac{\sqrt{nh_n}}{\sigma\|K^*\|_2} d_\infty-\ell_n\Big) + o_\pr(1).
			\nonumber 
		\end{align}
		Next, observe that the first term on the right-hand side of the previous display converges to a Gumbel distribution by \eqref{eq:convProkschStat}, whereas the second term diverges to $-\infty$ since 
		\begin{equation}\label{eq:ell_nrate}\ell_n=\sqrt{nh_n} \frac{\ell_n}{n^{1/4}h_n^{1/2}} n^{-1/4} =o(\sqrt{nh_n})
		\end{equation}
		by the assumptions on $h_n$. In the same way, it may be shown that
		\[
		\tfrac{\ell_n\sqrt{nh_n}}{\sigma\|K^*\|_2} \sup_{t\in\Ec^-\cap I_n}\Big\{  \tilde{\mu}_{h_n}(t)-g(\mu)  \Big\}  \to - \infty
		\]
		in probability. 

		   Now, note that  $D_n(\Ec)=\max\{D_{n,1}(\Ec), D_{n,2}(\Ec)\}$,  where 
		\begin{multline*}
		D_{n,1}(\Ec)= \tfrac{\ell_n \sqrt{nh_n}}{\sigma\|K^*\|_2}\max\Big\{ \sup_{t\in\Ec^+\cap I_n} \big ( 
		\tilde{\mu}_{h_n}(t)-g(\mu)	-d_\infty \big ) , \\
		 \sup_{t\in\Ec^-\cap I_n} \big ( -\tilde{\mu}_{h_n}(t)+g(\mu)-d_\infty \big )  \Big\}-\ell_n^2  
		\end{multline*}
		and
				\begin{multline*}
				D_{n,2}(\Ec)= \tfrac{\ell_n \sqrt{nh_n}}{\sigma\|K^*\|_2}\max\Big\{ \sup_{t\in\Ec^+\cap I_n}
		\big ( 
		-\tilde{\mu}_{h_n}(t)+g(\mu)-d_\infty \big ) , \\ \sup_{t\in\Ec^-\cap I_n} \big (	\tilde{\mu}_{h_n}(t)-g(\mu)-d_\infty 	\big )\Big\}-\ell_n^2. 
				\end{multline*}
	By the previous considerations, $D_{n,2}(\Ec)$ diverges to $-\infty$, thus 
	\begin{align}
	\label{eq:dndn1}
	D_n(\Ec) = D_{n,1}(\Ec) + o_\pr(1).
	\end{align}
	Next, observe that by similar arguments as in \eqref{eq:muHatPos}
		\begin{align*}
		D_{n,1}(\Ec)
		&= \tfrac{\ell_n \sqrt{nh_n}}{\sigma\|K^*\|_2}\max\Big\{ \sup_{t\in\Ec^+\cap I_n} \big ( 
		\tilde{\mu}_{h_n}(t)-\mu(t) \big ) , \sup_{t\in\Ec^-\cap I_n} \big (  -\tilde{\mu}_{h_n}(t)+\mu(t)\big )  \Big\}-\ell_n^2\\
		&= \ell_n \max\Big \{ \sup_{t\in\Ec^+\cap I_n}Z_{n,3}(t), \sup_{t\in\Ec^-\cap I_n}-Z_{n,3}(t) \Big\}  -\ell_n^2+o_\pr(1).
		\end{align*}
		The sets  $\Ec^+$ and $\Ec^-$ are disjoint and bounded away from each other. Thus, there exists a positive integer $n_0\in\N$ such that the arguments of the maximum in the previous display are stochastically independent for all $n\geq n_0$, and 
		\begin{align*}
		D_{n,1}(\Ec) &\stackrel{\Dc}{=} \ell_n \sup_{t\in\Ec\cap I_n}Z_{n,3}(t)  -\ell_n^2 + o_\pr(1)\\
		&= \frac{\ell_n}{\ell_n(\Ec)}\bigg( \ell_n(\Ec)\sup_{t\in\Ec\cap I_n}Z_{n,3}(t)-\ell_n^2(\Ec)\bigg) + \ell_n(\Ec)\ell_n -\ell_n^2 + o_\pr(1),
		\end{align*}
		 where
		\begin{align}\label{eq:lnEln-ln2} 
		\ell_n(\Ec)\ell_n -\ell_n^2  
		&= \frac{\ell_n}{\ell_n+\ell_n(\Ec)}(\ell_n^2(\Ec)-\ell_n^2) \nonumber \\ 
		&= 
		\frac{\sqrt{\log\big(\frac{(x_1-x_0)\Lambda_{K}}{2\pi h_n}\big)}}{\sqrt{\log\big(\frac{\lambda(\Ec)\Lambda_{K}}{2\pi h_n}\big)}+\sqrt{\log\big(\frac{(x_1-x_0)\Lambda_{K}}{2\pi h_n}\big)}} 2\bigg(\log\Big(\tfrac{\lambda(\Ec)\Lambda_{K}}{2\pi h_n}\Big)-\log\Big(\tfrac{(x_1-x_0)\Lambda_{K}}{2\pi h_n}\Big)\bigg) 
		\end{align}
		converges to $\log\Big(\tfrac{\lambda(\Ec)}{x_1-x_0}\Big)$ since
		\[  \frac{\sqrt{\log\big(\frac{(x_1-x_0)\Lambda_{K}}{2\pi h_n}\big)}}{\sqrt{\log\big(\frac{\lambda(\Ec)\Lambda_{K}}{2\pi h_n}\big)}+\sqrt{\log\big(\frac{(x_1-x_0)\Lambda_{K}}{2\pi h_n}\big)}} \to \frac{1}{2}. \]
		Weak convergence of $D_{n,1}(\Ec)$  to $\Gum_{\log\{ \lambda(\Ec)/(x_1-x_0) \} }$
then follows from \eqref{eq:ratioLnLnE}, \eqref{eq:lnEln-ln2}, and Remark \ref{stationary}(ii), and this implies the assertion by \eqref{eq:dndn1}.		\end{proof}

	\begin{proof}[Proof of of Proposition \ref{thm:Rn}]
		We only carry out the proof in the case that $\lambda(\Ec^+)>0$ and $\lambda(\Ec^-)>0$. For $\Ac\subset I_n$, let
		\begin{equation}\label{def:DnA} 
		D_n(\Ac)= \tfrac{\ell_n \sqrt{nh_n}}{\sigma\|K^*\|_2}\big(\sup_{t\in\Ac}|\tilde{\mu}_{h_n}(t)-g(\mu)|-d_\infty\big)-\ell_n^2
		\end{equation}
		and let $\Ec_n=\Ec_n^+\cup\Ec_n^-$ as in \eqref{eq:defE}, for some positive sequence $\rho_n$ that converges to $0$ slowly enough to guarantee that $\ell_n\sqrt{nh_n}\rho_n\to\infty$. Note that $D_n(\Ac)\leq D_n(\Bc)$ for $\Ac\subset \Bc$, and that there is an $n_0\in\N$ such that the sets   $\Ec_n^+$ and $\Ec_n^-$ are disjoint and bounded away from each other for $n\geq n_0$. With this notation, we can rewrite 
		\[
		\Rc_n= D_n - D_n(\Ec)=\max\big\{D_n(\Ec_n)-D_n(\Ec),D_n(I_n\setminus\Ec_n)-D_n(\Ec) \big\}.\]
		For the second term in the  maximum observe that
		\begin{align*}
		&\phantom{{}={}} D_n(I_n\setminus\Ec_n)-D_n(\Ec)\\
		&= \tfrac{\ell_n \sqrt{nh_n}}{\sigma\|K^*\|_2}\Big\{\sup_{t\in I_n\setminus\Ec_n} |\tilde{\mu}_{h_n}(t)-g(\mu)|-\sup_{t\in\Ec}|\tilde{\mu}_{h_n}(t)-g(\mu)| \Big\}\\
		&\leq \tfrac{\ell_n \sqrt{nh_n}}{\sigma\|K^*\|_2}\Big\{\sup_{t\in I_n\setminus\Ec_n} \Big(|\tilde{\mu}_{h_n}(t)-g(\mu)|-|\mu(t)-g(\mu)|\Big)\\
		&\hspace{2cm}+\sup_{t\in I_n\setminus\Ec_n}\Big(|\mu(t)-g(\mu)|-d_\infty\Big)-\sup_{t\in\Ec}\Big(|\tilde{\mu}_{h_n}(t)-g(\mu)|-d_\infty\Big) \Big\}\\
		&\leq \tfrac{\ell_n \sqrt{nh_n}}{\sigma\|K^*\|_2}\Big\{\sup_{t\in I_n\setminus\Ec_n} |\tilde{\mu}_{h_n}(t)-\mu(t)|-\rho_n-\sup_{t\in\Ec}\Big(|\tilde{\mu}_{h_n}(t)-g(\mu)|-d_\infty\Big) \Big\}\\
		&\leq  \tfrac{\ell_n \sqrt{nh_n}}{\sigma\|K^*\|_2}\sup_{t\in I_n} |\tilde{\mu}_{h_n}(t)-\mu(t)|- \ell_n^2 - \tfrac{\ell_n \sqrt{nh_n}}{\sigma\|K^*\|_2} \rho_n 	-D_n(\Ec)
		\end{align*} 
		which diverges to $-\infty$ since $D_n(\Ec)\convw \Gum_{\log\{ \lambda(\Ec)/(x_1-x_0) \} }$ by Proposition~\ref{thm:DnE}, 
		\[ 
		\tfrac{\ell_n \sqrt{nh_n}}{\sigma\|K^*\|_2}\sup_{t\in I_n} |\tilde{\mu}_{h_n}(t)-\mu(t)|-\ell_n^2 \convw \Gum_{\log(2)},
		\] 
		according to \eqref{eq:convProkschAbsStat}, and $\ell_n \sqrt{nh_n}\rho_n \to \infty$ by assumption. 
		
		As a consequence, it is sufficient to prove that $D_n(\Ec_n)-D_n(\Ec)=o_\Prob(1)$.  
		Suppose we have shown that
		\begin{align} \label{eq:rnpm}
		\Rc_n^{\pm}=D_n(\Ec_n^{\pm})-D_n(\Ec^{\pm})=o_\Prob(1).
		\end{align}
		Then, since $D_n(\Ec_n)-D_n(\Ec)=\max\{D_n(\Ec_n^+)-D_n(\Ec),D_n(\Ec_n^-)-D_n(\Ec)\}\leq \max\{\Rc_n^+, \Rc_n^-\}$,  the proof of the proposition is finished.
		
		For the proof of \eqref{eq:rnpm}, we only consider $\Rc_n^+$, as $\Rc_n^-$ can be treated similarly.
		Similar to \eqref{eq:muHatPos}, it holds
		\begin{align}\label{eq:muHat} 
		 \tfrac{\ell_n \sqrt{nh_n}}{\sigma\|K^*\|_2} \sup_{t\in\Ec_n^+} -\big(\tilde{\mu}_{h_n}(t)-g(\mu)\big) 
		 &= 
		 \ell_n \sup_{t\in\Ec_n^+}\Big\{  - Z_{n,3}(t) - \tfrac{\sqrt{nh_n}}{\sigma\|K^*\|_2}\big(\mu(t)-g(\mu)\big)\Big\} + o_\pr(1)\nonumber  \\
		&\leq \ell_n\sup_{t\in\Ec_n^+}\{ - Z_{n,3}(t)\} -\tfrac{\ell_n\sqrt{nh_n}}{\sigma\|K^*\|_2}(d_\infty-\rho_n) + o_\pr(1) \nonumber  \\
		&\stackrel{\Dc}{=} \ell_n\sup_{t\in\Ec_n^+}Z_{n,3}(t) -\tfrac{\ell_n\sqrt{nh_n}}{\sigma\|K^*\|_2}(d_\infty-\rho_n) + o_\pr(1) \nonumber  \\
		&\leq \ell_n\sup_{t\in I_n}Z_{n,3}(t)-\ell_n^2 -\ell_n\big\{\tfrac{\sqrt{nh_n}}{\sigma\|K^*\|_2}(d_\infty-\rho_n)-\ell_n\big\} + o_\pr(1).
		\end{align}
		The first term on the right-hand side converges by Remark \ref{stationary}(ii) to a Gumbel distribution, whereas the second term diverges to $-\infty$ since $\ell_n/\sqrt{nh_n}\to 0$ by \eqref{eq:ell_nrate}. Hence,
		\[
		\tfrac{\ell_n \sqrt{nh_n}}{\sigma\|K^*\|_2} \sup_{t\in\Ec_n^+} -\big(\tilde{\mu}_{h_n}(t)-g(\mu)\big) \to -\infty
		\]
		in probability,
		and by monotonicity the same is true if $\Ec_n^+$ is replaced by $\Ec^+$.	Thus, by the definition of $D_n(\Ec_n^+)$ in \eqref{def:DnA}, we have 
		\begin{align*}
		\Rc_n^+ 
		&=
		\tfrac{\ell_n\sqrt{nh_n}}{\sigma\|K^*\|_2}\Big\{ \sup_{t\in\Ec_n^+} | \tilde{\mu}_{h_n}(t) - g(\mu)|  -\sup_{t\in\Ec^+} | \tilde{\mu}_{h_n}(t)- g(\mu)| \Big\} \\
		&=
		\tfrac{\ell_n\sqrt{nh_n}}{\sigma\|K^*\|_2}\Big\{ \sup_{t\in\Ec_n^+}\tilde{\mu}_{h_n}(t)-\sup_{t\in\Ec^+}\tilde{\mu}_{h_n}(t)\Big\}+o_\pr(1).
		\end{align*}
		
		Next, for $\theta\ge0$, let $U_\theta(\Ac)$ denote the $\theta$-neighbourhood of $\Ac\subset I_n$ in $I_n$. Define 
		\begin{equation*}
		\delta_n = 2 \inf\{ \theta\geq 0 \mid \Ec_n^+ \subset U_\theta(\Ec^+) \}.
		\end{equation*} 
		We proceed by showing that $\delta_n=o(1)$ for $n\to\infty$. As $\Ec_n^+$ is a descending sequence of sets, the nonnegative sequence $\delta_n$ decreases and therefore 
		converges. Suppose $\delta_n>0$ for all $n$. Then, by definition, $\Ec_n^+\subset U_{\delta_n}(\Ec^+)$ but $\Ec_n^+\not\subset U_{\delta_n/4}(\Ec^+)$. Thus, there exists a sequence $t_n\in\Ec_n^+\subset[0,1]$ such that $|t_n-t|\geq \delta_n/4$ for all $t\in\Ec^+$ and all $n\in\N$. By compactness of $[0,1]$ and continuity of $\mu$, there is a convergent subsequence $(t_{n_k})_{k\in\N}$ with $\lim\limits_{k\to\infty}t_{n_k}=t^*\in[0,1]$ and $\mu(t^*)-g(\mu)=\lim\limits_{k\to\infty}\mu(t_{n_k})-g(\mu)=d_\infty$, thus $t^*\in\Ec^+$. Hence, $|t_{n_k}-t^*|\geq \delta_{n_k}/4$, which implies $\delta_n=o(1)$ as asserted.
		
		Since $\Ec_n^+\subset U_{\delta_n}(\Ec^+) \cap I_n$ and since $\Ec^+ \cap I_n \subset \Ec^+$, the assertion $\Rc_n^+=o_\Prob(1)$ and hence the proposition follows from
		\begin{align}
		\label{eq:rn2}
		\tfrac{\ell_n\sqrt{nh_n}}{\sigma\|K^*\|_2}\Big\{\sup_{t\in U_{\delta_n}(\Ec^+)\cap I_n}\tilde{\mu}_{h_n}(t)-\sup_{t\in\Ec^+\cap I_n}\tilde{\mu}_{h_n}(t)\Big\} =o_\Prob(1).
		\end{align}
		For the proof of \eqref{eq:rn2} observe that,
		\begin{align}\label{de1}
		 &\phantom{{}={}}\tfrac{\sqrt{nh_n}}{\sigma\|K^*\|_2}\Big(\sup_{t\in U_{\delta_n}(\Ec^+)\cap I_n}\tilde{\mu}_{h_n}(t)-\sup_{t\in\Ec^+\cap I_n}\tilde{\mu}_{h_n}(t)\Big) \\
		&= 
		\max\Big\{0, \tfrac{\sqrt{nh_n}}{\sigma\|K^*\|_2}\Big(\sup_{t\in (U_{\delta_n}(\Ec^+)\setminus \Ec^+)\cap I_n}\tilde{\mu}_{h_n}(t)-\sup_{t\in\Ec^+\cap I_n}\tilde{\mu}_{h_n}(t)\Big)\Big\}\nonumber \\
		&\leq 
		\max\Big\{0, \tfrac{\sqrt{nh_n}}{\sigma\|K^*\|_2}\Big(\sup_{t\in (U_{\delta_n}(\Ec^+)\setminus \Ec^+)\cap I_n}\{\tilde{\mu}_{h_n}(t)-\mu(t)\}-\sup_{t\in\Ec^+\cap I_n}\{\tilde{\mu}_{h_n}(t)-\mu(t)\}\Big)\Big\} \nonumber \\
		&= 
		\max\Big\{0,\sup_{t\in (U_{\delta_n}(\Ec^+)\setminus \Ec^+)\cap I_n} Z_{n,3}(t)-\sup_{t\in\Ec^+\cap I_n}Z_{n,3}(t)\Big\}+o_\pr(|\log h_n|^{-1/2}), 	\nonumber
		\end{align} 
		by  \eqref{eq:zn1zn3}, where the process $Z_{n,3}$ is defined in \eqref{hd30}.
		
		By Remark~\ref{rem:concavity} and since $\lambda(\Ec^+)>0$, $\Ec^+$ can be rewritten as $\bigcup_{j=1}^m [x_{j,1},x_{j,2}]$, for some $m\in\N$, where at least one interval has positive length. Without loss of generality, let $x_{j,1}<x_{j,2}$, for $1\leq j\leq m_1$, and $x_{j,1}=x_{j,2}$, for $m_1+1\leq j\leq m$. There exists an integer  $n_0\in\N$ such that for any $n\geq n_0$, the intervals $I_{j,1}=(x_{j,1}-\delta_n-h_n, x_{j,1}+\delta_n+h_n)$ and $I_{j,2}=(x_{j,2}-\delta_n-h_n, x_{j,2}+\delta_n+h_n)$, for $1\leq j\leq m_1$, and $I_{j,3}=(x_{j,1}-\delta_n-h_n, x_{j,2}+\delta_n+h_n)$, for $m_1+1\leq j\leq m$, are disjoint with pairwise distance greater than $2h_n$. By definition, the kernel $K^*$ vanishes outside of the interval $[-1,1]$, hence, observing the definition of $Z_{n,3}$ in \eqref{hd30}, the supremum over the set $(U_{\delta_n}(\Ec^+)\setminus \Ec^+) \cap I_n$ in \eqref{de1} can be bounded by
		\begin{equation}\label{HolgerDette}
		Y_n =  \max\big\{\max_{j=1}^{m_1} R_{j,n,1},\max_{j=1}^{m_1} R_{j,n,2},\max_{j=m_1+1}^{m} R_{j,n,3}\big\}, \end{equation}
		where the random variables
		$$
		R_{j,n,k}= \sup_{t\in I_{j,k}}\frac{1}{\|K^*\|_2 \sqrt{nh_n}}\sum_{i\in\Z} V_{i} K^*_{h_n}(\tfrac{i}{n}-t)
		$$ 
		are all independent (note that in each term the range of summation is given by $i\in \{\lceil nt-nh_n\rceil,\dots, \lfloor nt+nh_n\rfloor\}$, where 
		$t \in  I_{j,k}$). 
		
		Next, consider the second supremum on the right-hand side of \eqref{de1}, which can be bounded below by restricting the supremum to $[x_{1,1}, x_{1,2}]$, where, by assumption $x_{1,1}< x_{1,2}$. Next, one may choose  $\lfloor\tfrac{x_{1,2}-x_{1,1}}{2\delta_n+4h_n}\rfloor-2$ disjoint intervals of length $2\delta_n+2h_n$ with distance $2h_n$ contained in the interval $[x_{1,1},x_{1,2}]$, that have a distance of at least $2h_n$ to the intervals $I_{1,1}$ and $I_{1,2}$. We may hence bound the second supremum in \eqref{de1} from below by
		\[
	Y_n'=\max_{j=1}^{\lfloor\tfrac{x_{1,2}-x_{1,1}}{2\delta_n+4h_n}\rfloor-2} R'_{j,n}, 
		 \]
		where $R'_{1,n}, R'_{2,n}, \dots$ are independent with
$$
 R'_{j,n}  \stackrel{\Dc}{=}	\sup_{t\in(-\delta_n-h_n,\delta_n+h_n)}\frac{1}{\|K^*\|_2 \sqrt{nh_n}}\sum_{i\in\Z} V_{i} K^*_{h_n}(\tfrac{i}{n}-t) \stackrel{\Dc}{=} R_{j,n,k},
 $$
 and also independent from the random variables $R_{j,n,k}$.  Using the notation $Y'_{j,n}=\max_{k=1}^{m+m_1} R'_{(j-1)(m+m_1)+k,n}$ for $j=1,\ldots, \nu_n$, where
 \[
 \nu_n = \lfloor\tfrac{(x_{1,2}-x_{1,1})-2(2\delta_n+4h_n)}{(m+m_1)(2\delta_n+4h_n)}\rfloor,
 \] 
 we may write $Y'_n= \max_{j=1}^{\nu_n} Y'_{j,n}$. Note that $Y'_{1,n}, Y'_{2,n}, \dots$ are independent. Then,  using the notation \eqref{HolgerDette} 
 \begin{align*}
		&\phantom{{}={}} \pr\Big( \sup_{t\in (U_{\delta_n}(\Ec^+)\setminus \Ec^+)\cap I_n} Z_{n,3}(t) > \sup_{t\in\Ec^+\cap I_n}Z_{n,3}(t) \Big) \\
		&\le 
		\pr\Big(\max\Big\{\max_{j=1}^{m_1} R_{j,n,1},\max_{j=1}^{m_1} R_{j,n,2},\max_{j=m_1+1}^{m} R_{j,n,3}\Big\} > \max_{j=1}^{\lfloor\tfrac{x_{1,2}-x_{1,1}}{2\delta_n+4h_n}\rfloor-2} R'_{j,n}\Big)\\
		&\le 
		\pr\Big(Y_n>\max_{j=1}^{\nu_n}Y'_{j,n}\Big)\\
		&= \int_{-\infty}^\infty f_{Y_n}(y) \prod_{j=1}^{\nu_n} \pr(Y'_{j,n}<y)\diff y
		= 
		\ex\Big[F_{Y_n}^{\nu_n}(Y_n)\Big]
 		= \frac{1}{\nu_n+1}
		= \Oc(\delta_n+h_n),
		\end{align*}
		where $F_{Y_n}$ and $f_{Y_n}$ denote the cumulative distribution and the density function of $Y_n$, respectively. Here the last estimate follows since $F_{Y_n}(Y_n)$ is uniformly distributed on $[0,1]$.
		 Thus, the right-hand side of \eqref{de1} converges to 0 in probability, which implies the proposition.
	\end{proof}

\begin{proof}[Proof of   Corollary \ref{FloZirkus}.]

We only prove the case (2) - the first case follows by similar arguments.
Observe that 
\begin{align*}
&\phantom{{}={}} \pr\bigg( \hat d_{\infty, n} > (q_{0,1-\alpha}+\ell_n^2)\frac{\hat{\sigma}\|K^*\|_2}{\sqrt{nh_n}\ell_n}+\Delta \bigg)\\
&= 
\pr\bigg( \frac{\sqrt{nh_n}\ell_n}{\hat\sigma\|K^*\|_2} \big( \hat d_{\infty, n} - d_{\infty, n } \big)-\ell_n^2 
+ \frac{\sqrt{nh_n}\ell_n}{\hat\sigma\|K^*\|_2}( d_{\infty, n} - d_\infty + d_\infty - \Delta) > q_{0,1-\alpha}\bigg).
\end{align*}
In case $\lambda(\Ec)>0$, 
the right-hand side converges to $0$ for $d_\infty < \Delta$, to $1$ for $d_\infty > \Delta$, and its limit can be bounded from above by 
\[
\alpha^* =\Gum_{\log\{ \lambda(\Ec)/(x_1-x_0) \} }((q_{0,1-\alpha,\infty})) \le 
\Gum_{0}((q_{0,1-\alpha,\infty})) = \alpha
\] 
for $d_\infty=\Delta$, by Theorem \ref{thm:Dn} and \eqref{eq:estLRV}. 

For $\lambda(\Ec)=0$, observing Assumption \ref{assump:ghat}, \eqref{eq:jackKnife} and the inverse triangle  inequality, we have  
\begin{align}\label{eq:axel4}
|\hat{d}_{\infty,n}-d_{\infty,n}| 
&= \Big| \sup_{t\in I_n}|\tilde{\mu}_{h_n}(t)-\hat{g}_n|-\sup_{t\in I_n}|\mu(t)-g(\mu)| \Big| \nonumber \\ 
&\leq \sup_{t\in I_n}|\tilde{\mu}_{h_n}(t)-\mu(t)|+o_\pr(1)
= \sup_{t\in I_n}\bigg|\frac{1}{nh_n}\sum_{i=1}^{n}\sigma\big(\tfrac{i}{n}\big)V_i K_{h_n}^*\big(\tfrac{i}{n}-t\Big) \bigg| +o_\pr(1).
\end{align}
By the same arguments that led to \eqref{hd3a}, the right-hand side of the previous display is of order $o_\pr(1)$, thus, $|\hat{d}_{\infty,n}-d_{\infty,n}|=o_\pr(1)$.
Now, under the alternative,
\begin{align*}
p_n & \equiv \pr\bigg( \hat d_{\infty, n} > (q_{0,1-\alpha}+\ell_n^2)\frac{\hat{\sigma}\|K^*\|_2}{\sqrt{nh_n}\ell_n}+\Delta \bigg)\\
&= 
\pr\bigg( \frac{\sqrt{nh_n}\ell_n}{\hat\sigma\|K^*\|_2} \bigg\{ \hat d_{\infty, n} - d_{\infty, n }+ d_{\infty, n} - d_\infty + d_\infty - \Delta - \frac{\ell_n \hat{\sigma} \|K^*\|_2}{\sqrt{nh_n}}\bigg\} > q_{0,1-\alpha}\bigg),
\end{align*}
converges to $1$ since $(\hat d_{\infty, n} - d_{\infty, n }), (d_{\infty, n} - d_\infty)$ and $\ell_n(nh_n)^{-1/2}$ vanish as $n$ tends to infinity, while $d_\infty - \Delta>0$. The assertion regarding the null hypothesis follows from Theorem \ref{thm:Dn}(3). 
\end{proof}

\subsection{Proofs for Section \ref{sec4}}  \label{proofsec4} 

We first prove an auxiliary result, which
will be used in the proof of  Theorem \ref{thm:EnEstimator}.

\begin{proposition}\label{prop:conv:quot}
	Let the assumptions of Theorem \ref{thm:EnEstimator} and Assumption \ref{assump:mu} be satisfied. Then,  
	\[ 
	\frac{\lambda(\Ec_{n,1}^\pm)}{\lambda\big(\Ec_{n,2}^\pm\big)}\to 1 \qquad \text{ if } \Ec^{\pm} \ne \emptyset,
		\]
	where $\Ec_{n,1}^\pm=\Ec_n^{\pm}$ are defined in \eqref{eq:defE} and where 
	\begin{equation}\label{eq:defEn2}
	\Ec_{n,2}^{\pm}=\Big \{ t\in I_n: d_{\infty, n}  \mp d(t)\leq (1+2e_n)\rho_n \Big  \},
	\end{equation}
	with $e_n$ an arbitrary positive sequence converging to $0$.
\end{proposition}

\begin{proof} We only consider the assertion for $\Ec^{+}\ne \emptyset$.
		As both the numerator and the denominator converge from above to $\lambda(\Ec^+)$, the convergence is trivial for $\lambda(\Ec^+)>0$. In the case $\lambda(\Ec^+)=0$, by Assumption \ref{assump:mu}, there 
		exists  a positive constant $\gamma>0$ such that the function $d=\mu-g(\mu)$ is concave on $U_\gamma(\Ec^+)$.  
Let us show that $\Ec_{n,2}^+\subset U_\gamma(\Ec^+)$ for all sufficiently large $n$. First, by continuity, $d$ attains its maximum on the compact set $[0,1]\setminus U_{\gamma}(\Ec^+)$, say at a point $t'$. As $t'\notin \Ec^+$, $d(t')<d_\infty$. Next, $\sup_{t\in\Ec_{n,2}^+}d_\infty-d(t)=(1+2e_n)\rho_n$ by definition of $\Ec_{n,2}^+$ and continuity of $d$, which converges to 0 for $n\to\infty$. Thus, there is a natural number $n_0\in\N$ such that for any $n \geq n_{0}$ and any $t\in \Ec_{n,2}^+$, it holds $d(t)>d(t')$, and hence $t\notin[0,1]\setminus U_\gamma(\Ec^+)$. Therefore, it follows $\Ec_{n,2}^+\subset U_\gamma(\Ec^+)$.

	Since $\lambda(\Ec^+)=0$, we can  use part (ii) of Remark \ref{rem:concavity}  and rewrite $\Ec^+$
	 as $\{t_1,\dots,t_m\}$, for some $m\in\N$. By Assumption \ref{assump:mu},  $\mu-g(\mu)$ is strictly increasing on $U_\gamma^<(t_i):=(t_i-\gamma, t_i]$ and strictly decreasing on $U_\gamma^>(t_i):=[t_i,t_i+\gamma)$, for any $t_i\in \Ec^+$. Thus, for sufficiently large $n$, there are points $x_{i,\sim,1}>x_{i,\sim,2}\in U_\gamma^\sim(t_i)$ (depending on $n$) with $\mu(x_{i,\sim,1})-g(\mu)=d_\infty-\rho_n$ and $\mu(x_{i,\sim,2})-g(\mu)=d_\infty-(1+2e_n)\rho_n$, where the symbol $\sim$ denotes either '$<$' or '$>$'. Moreover, these numbers are unique.
	As the function $\mu$ is concave on the interval $[x_{i,<,2},t_i]$, we have 
	\[\frac{\rho_n}{t_i-x_{i,<,1}}=\frac{\mu(t_i)-\mu(x_{i,<,1})}{t_i-x_{i,<,1}}\leq \frac{\mu(x_{i,<,1})-\mu(x_{i,<,2})}{x_{i,<,1}-x_{i,<,2}}=\frac{2e_n\rho_n}{x_{i,<,1}-x_{i,<,2}}. \]
	Thus, $x_{i,<,1}-x_{i,<,2} \leq 2e_n(t_i-x_{i,<,1})$ and analogously $x_{i,>,2}-x_{i,>,1}\leq 2e_n(x_{i,>,1}-t_i)$. Since $\Ec_{n,k}^+=\bigcup_{i=1}^m (x_{i,<,k},t_i] \cup (t_i,x_{i,>,k})$, for $k=1,2$, it follows that
	\begin{align*}
	\frac{\lambda(\Ec_{n,2}^+\setminus\Ec_{n,1}^+)}{\lambda(\Ec_{n,1}^+)}&=\sum_{i=1}^m \frac{x_{i,<,1}-x_{i,<,2}+x_{i,>,2}-x_{i,>,1}}{\lambda(\Ec_{n,1}^+)}\\ 
	&\leq \sum_{i=1}^m 2e_n\frac{t_i-x_{i,<,1}+x_{i,>,1}-t_i}{\lambda(\Ec_{n,1}^+)}=2e_n =o(1),
	\end{align*} 
	which implies also $\lambda\big(\Ec_{n,2}^+\big)/\lambda\big(\Ec_{n,1}^+\big)\to 1$.
\end{proof}

\begin{proof}[Proof of Theorem \ref{thm:EnEstimator}]
First, consider the case $d_\infty>0$. Choose a positive sequence $e_n$ converging to 0 such that $\liminf_{n\to\infty} e_n\rho_n\sqrt{nh_n}/2-\ell_n >0$. We prove that
	\begin{equation}\label{eq:EnEstimatorConv} 
	\frac{\lambda(\Ec_n)}{\lambda(\hat{\Ec}_{n,2})} =1+o_\Prob(1) ,
	\end{equation}
	where $\hat{\Ec}_{n,2}=\hat\Ec_{n,2}^+\cup\hat\Ec_{n,2}^-$ and
	\[
	\hat{\Ec}_{n,2}^{\pm}=\{ t\in I_n: \hat d_{\infty, n} \mp \hat{d}_n(t)\leq (1+e_n)\rho_n \}.
	\]
	By the same arguments one may show that
	$\lambda(\Ec_n)/\lambda(\hat{\Ec}_{n,3})=1+o_\Prob(1)$,
	where $\hat{\Ec}_{n,3}$ is defined analogously to $\hat{\Ec}_{n,2}$ with $(1+e_n)\rho_n$ replaced by $(1-e_n)\rho_n$.
	As $\hat{\Ec}_{n,3}\subset \hat{\Ec}_n\subset\hat{\Ec}_{n,2}$ this leads to $\lambda(\Ec_n)/\lambda(\hat{\Ec}_n)=1+o_\Prob(1)$ by the sandwich theorem and proves the 
	assertion of the theorem.
	
	Without loss of generality, we assume that $\Ec^+\neq \emptyset$  (which implies $\lambda(\Ec_n^+)>0$) and observe that
\begin{align}\label{eq:axel2}
   \phantom{{}={}}2\sup_{t\in I_n}|\hat d_n(t) - d(t)|  \nonumber
   &\ge
	 | \hat d_{\infty, n} - d_{\infty, n} | + \sup_{t\in I_n}\big|\hat{d}_n(t)-d(t)\big| \nonumber\\
	&\geq | \hat d_{\infty, n} - d_{\infty, n} | + \sup_{t\in\Ec^+_n} \big|\hat{d}_n(t)-d(t)\big| \nonumber\\
	&\geq |\hat d_{\infty,n} - d_{\infty, n}| + \sup_{t\in\Ec^+_n}  \big| \hat{d}_n(t) - d_{\infty, n}\big| -\rho_n \nonumber \\
	&\geq \sup_{t\in\Ec^+_n} \big| \hat d_{\infty,n} -\hat{d}_n(t)\big| -\rho_n.
	\end{align}		
	Thus 
	\begin{align}\label{eq:axel3}
		\pr(\Ec^+_n\subset \hat{\Ec}^+_{n,2})
		&=  \pr\Big(\sup_{t\in\Ec^+_n} \big\{  \hat{d}_{\infty, n} - \hat d_n(t)  \big\} \le (1+e_n)\rho_n\Big)  \nonumber\\
		&=  \pr\Big(\sup_{t\in\Ec^+_n} \big|  \hat{d}_{\infty, n} - \hat d_n(t)  | - \rho_n \le e_n\rho_n\Big) \nonumber \\
		&\geq \pr\Big(2\sup_{t\in I_n}|\hat d_n(t) - d(t)| \leq e_n\rho_n\Big) \nonumber \\
		&=  \pr\Big( \sqrt{nh_n} \ell_n \sup_{t\in I_n}|\hat d_n(t) - d(t)|- \ell_n^2 \leq e_n\rho_n\ell_n\sqrt{nh_n}/2-\ell_n^2 \Big) 
	\end{align}	
	which converges to unity by Remark~\ref{stationary}(ii), by Assumption~\ref{assump:ghat} and since $e_n\rho_n\ell_n\sqrt{nh_n}/2-\ell^2_n\to\infty$ by the choice of $e_n$. Note that 
	\[
	\sup_{t\in \hat{\Ec}^+_{n,2}}|\hat d_{\infty,n}-\hat{d}_n(t)|=(1+e_n)\rho_n,
	\] 
	thus, analogously to \eqref{eq:axel2}, 
	\begin{align*}
	&\phantom{{}={}} \sup_{t\in \hat{\Ec}^+_{n,2}}|d_{\infty,n}-d(t)|-(1+e_n)\rho_n \\ 
	& \le | d_{\infty, n} - \hat d_{\infty, n}| + \sup_{t\in \hat{\Ec}^+_{n,2}} |\hat d_{\infty, n} - \hat d_n(t) | +  \sup_{t\in \hat{\Ec}^+_{n,2}}  |\hat d_n(t) - d(t)| - (1+e_n)\rho_n  \\
	&\le 2\sup_{t\in I_n}|\hat d_n(t) - d(t)|.
	\end{align*}
	Therefore, as in \eqref{eq:axel3}, $\pr(\hat{\Ec}^+_{n,2}\subset \Ec^+_{n,2})\to 1$, where the set $\Ec^+_{n,2}$ is defined in \eqref{eq:defEn2}. In particular,
	\[
	\Ec^+_n \subset \hat{\Ec}^+_{n,2}\subset \Ec^+_{n,2},
	\]
	with probability converging to one. Hence, 
	\begin{equation} \label{eq:EnEstimator}
	0 <  \lambda(\Ec^+_n)  \leq \lambda(\hat{\Ec}^+_{n,2}) \leq \lambda(\Ec^+_{n,2}), 
	\end{equation}
	with probability converging to one and
	\[ 
	\pr\bigg( 1\leq \frac{\lambda(\hat{\Ec}^+_{n,2})}{\lambda(\Ec^+_n)} \leq z_n\bigg)\xrightarrow{n\to\infty} 1, 
	\]
	where the deterministic sequence $z_n=\lambda(\Ec^+_{n,2})/\lambda(\Ec^+_n)$ converges to $1$, by Proposition \ref{prop:conv:quot}. Let $A_n$ denote the event that \eqref{eq:EnEstimator} holds. Then, $\pr(A_n^C)\to 0$ and, for any $\eps>0$,
	\begin{align*}
	\pr\bigg(\bigg| \frac{\lambda(\hat{\Ec}^+_{n,2})}{\lambda(\Ec^+_n)}-1 \bigg|\geq \eps\bigg)
	&= 
	\pr\bigg(\bigg| \frac{\lambda(\hat{\Ec}^+_{n,2})}{\lambda(\Ec^+_n)}-1 \bigg|\geq \eps \cap A_n\bigg)
	+ \pr\bigg(\bigg| \frac{\lambda(\hat{\Ec}^+_{n,2})}{\lambda(\Ec^+_n)}-1 \bigg|\geq \eps \cap A_n^C\bigg)\\
	&\leq \pr\big(|z_n-1|\geq \eps\big)+\pr(A_n^C)
	\to 0.
	\end{align*}
	Therefore, it follows 
	\[\frac{\lambda(\Ec^+_n)}{\lambda(\hat{\Ec}^+_{n,2})}=1+o_\Prob(1).\]
	If $\Ec^-\neq \emptyset$, it holds analogously
	\[\frac{\lambda(\Ec^-_n)}{\lambda(\hat{\Ec}^-_{n,2})} =1+o_\Prob(1). \] 
Note that, since $d_\infty>0$, $\hat{\Ec}^+_{n,2}$ and $\hat{\Ec}^-_{n,2}$ with probability converging to one. Thus, the previous displays imply
	\begin{align*}
		\frac{\lambda(\Ec_n)}{\lambda(\hat{\Ec}_{n,2})} 
		&= 
		\frac{\lambda(\Ec^+_n)}{\lambda(\hat{\Ec}^+_{n,2})} \frac{\lambda(\hat{\Ec}^+_{n,2})}{\lambda(\hat{\Ec}_{n,2})} + \frac{\lambda(\Ec^-_n)}{\lambda(\hat{\Ec}^-_{n,2})} \frac{\lambda(\hat{\Ec}^-_{n,2})}{\lambda(\hat{\Ec}_{n,2})}  \\
	&= 
	\{ 1+o_\Prob(1) \} \Big\{ \frac{\lambda(\hat{\Ec}^+_{n,2})}{\lambda(\hat{\Ec}_{n,2})} + \frac{\lambda(\hat{\Ec}^-_{n,2})}  {\lambda(\hat{\Ec}_{n,2})} \Big\} 
	= 
	1+o_\Prob(1).
	\end{align*}
	Conversely, if $\Ec^-= \emptyset$, both $\Ec_n^-$ and $\Ec^-_{n,2}$ are  empty for almost every $n\in\N$. Further, $\pr(\hat{\Ec}^-_{n,2}\subset \Ec^-_{n,2})\to 1$, which implies
	\[ 
	\frac{\lambda(\hat{\Ec}_{n,2})}{\lambda(\Ec_n)}
	=
	\frac{\lambda(\hat{\Ec}_{n,2}^+)}{\lambda(\Ec_n^+)} + \frac{\lambda(\hat{\Ec}_{n,2}^-)}{\lambda(\Ec_n^+)} - \frac{\lambda(\hat{\Ec}_{n,2}^+ \cap \hat{\Ec}_{n,2}^-)}{\lambda(\Ec_n^+)} =1+o_\Prob(1)
	, \]
	and this is equivalent to \eqref{eq:EnEstimatorConv}.

Finally, if $d_\infty=0$, observe that $\Ec_n=I_n$ and
			\[
		\hat{\Ec}_{n,2}=\Big\{ t\in I_n: \Big|\sup_{s\in I_n}|\hat d_{\infty, n}(s)-d(s)|-\big(\hat{d}_n(t)-d(t)\big)\Big|\leq (1+e_n)\rho_n \Big\}.
		\]	
		By definition $\hat \Ec_{n,2}\subset I_n$, and further $$\pr(I_n\subset \hat{\Ec}_{n,2})\ge \pr(2\sup_{t\in I_n}|\hat{d}_n(t)-d(t)|\leq (1+e_n)\rho_n),$$ hence, the theorem's assertion follows by  similar arguments as in the case $d_\infty>0$.
\end{proof}

The proof of Theorem \ref{conv:DnEn} will be based on the following auxiliary result.

\begin{proposition}\label{prop:BlowingUp}
	Let Assumptions \ref{assump:kernel} and \ref{assump:mu} be satisfied, $\rho_n\to 0$ and $h_n\to 0$. Additionally, if $\lambda(\Ec)=0$ assume that $\rho_nh_n^{-2}\to\infty$. Then, $(h_n^{-1}\Ec_n)_n$ satisfies the blowing up property from \cite{piterbarg2012}:
	\begin{compactenum}[(i)] 
		\item $h_n^{-1}\lambda(\Ec_n)=\lambda(h_n^{-1}\Ec_n)\to \infty$.
		\item For any $R>1$, 
		\[\lambda(U_R(h_n^{-1}\Ec_n)\setminus (h_n^{-1}\Ec_n) )\leq 2R |\partial (h_n^{-1}\Ec_n)|,\]
		 where $|\Ac|$ and $\partial \Ac$ denote the cardinality and the boundary of a set $\Ac$.
		\item 
		For any $\alpha\in (0,1)$, there exists a constant $L=L_\alpha>0$ such that $|\partial (h_n^{-1}\Ec_n)| \leq L \big(\lambda(h_n^{-1}\Ec_n)\big)^\alpha$. 
	\end{compactenum}
\end{proposition}

\begin{proof} 
	(i) If $\lambda(\Ec)>0$, then $\lambda(h_n^{-1}\Ec_n)\geq h_n^{-1}\lambda(\Ec)\to\infty$. If $\lambda(\Ec)=0$, we assume without loss of generality that $\Ec^+\neq \emptyset$ and define 
	$c_n =(\frac{\rho_n}{\|\mu''\|_\infty})^{1/2}$. Then for  
	$t^*\in\Ec^+$ and $t\in U_{c_n}(t^*)$, it follows
	\[ d_\infty-d(t) =\mu(t^*)-\mu(t)\leq \|\mu''\|_\infty|t^*-t|^2 \leq \rho_n, \]
	by Taylor's theorem and $\mu'(t^*)=0$. Thus, $U_{c_n}(\Ec^+) \cap I_n \subset \Ec_n^+$. In particular, $\lambda(\Ec_n^+) \ge \lambda(U_{c_n}(\Ec^+) \cap I_n) \geq 2( c_n -  h_n)$. The assertion follows from $\sqrt{\rho_n}h_n^{-1}\to\infty$.
		
		\noindent
	(ii) For $t\in U_R(h_n^{-1}\Ec_n)\setminus (h_n^{-1}\Ec_n)$, there is an $s\in h_n^{-1}\Ec_n$ with distance $|s-t|<R$. Therefore, there is an intermediate value $u \in [s\wedge t, s\vee t]$ in $\partial( h_n^{-1}\Ec_n)$ such that $t\in U_R(u)$. Thus, 
	\[ 
	\lambda\Big(U_R(h_n^{-1}\Ec_n)\setminus h_n^{-1}\Ec_n \Big) \leq \lambda\Big(\bigcup_{u\in\partial (h_n^{-1}\Ec_n)}U_R(u) \Big) \leq 2R|\partial (h_n^{-1}\Ec_n)|, 
	\]
	and (ii) follows. 
	
	\noindent
	(iii) 	
	By Remark \ref{rem:concavity} (ii), the set of extremal points can be represented as $\Ec=\bigcup_{i=1}^m [x_{i,1},x_{i,2}]$, for some $m\in\N$.
	Further, there exists an integer $n_0\in\N$, such that for any $n\geq n_0$, $\Ec_n\subset U_\gamma(\Ec)$, with $\gamma$ from Assumption~\ref{assump:mu}. By concavity of $\mu$ on $U_\gamma([x_{i,1},x_{i,2}])$, for $[x_{i,1},x_{i,2}]\subset\Ec^+$ and convexity on $U_\gamma([x_{i,1},x_{i,2}])$, for $[x_{i,1},x_{i,2}]\subset\Ec^-$, there are real numbers $x_{i,1,n}\leq x_{i,1} \leq x_{i,2}\leq x_{i,2,n}$ such that $\Ec_n=\bigcup_{i=1}^m [x_{i,1,n},x_{i,2,n}]$. 
Hence, $ |\partial (h_n^{-1}\Ec_n)|\le2m$, which implies the assertion by (i).	
\end{proof}

\begin{proof}[Proof of Theorem \ref{conv:DnEn}.] If $d_\infty=0$, it follows that $\Ec_n=I_n$ and further $\ell_n(\Ec_n)=\ell_n$. Thus, the theorem's statement follows from the beginning of the proof of Theorem \ref{thm:Dn}(2). 

Now, let $d_\infty>0$.
	 	First observe that with the notation 
	 	\begin{equation*}
		D_n^\Ec(\Ac)=\ell_n(\Ec_n) \frac{\sqrt{nh_n}}{\sigma\|K^*\|_2}\Big\{\sup_{t\in\Ac}|\hat{d}_n(t)|- d_{\infty,n}\Big\} - \ell_n^2(\Ec_n), \end{equation*}
	 	for $\Ac\subset I_n$, we have 
		\[
		 D_n^\Ec(I_n)  = \tfrac{\ell_n(\Ec_n)\sqrt{nh_n}}{\sigma\|K^*\|_2} \big( \hat d_{\infty, n} - d_{\infty, n}\big)-\ell_n^2(\Ec_n) =D_n^\Ec(I_n\setminus \Ec_n)\vee D_n^\Ec(\Ec_n).
		\] 
		Let us first consider the case $\Ec_n \ne I_n$.
	 	Then, by the definition of $\Ec_n=\Ec_n^+ \cup \Ec_n^-$ in \eqref{eq:defE},
	 	\begin{align*}
	 	&\phantom{{}={}}D_n^\Ec\big(I_n\setminus \Ec_n\big)\\
	 	&\leq \ell_n(\Ec_n)\frac{\sqrt{nh_n}}{\sigma\|K^*\|_2} \Big\{\sup_{t\in I_n\setminus\Ec_n}\{|\hat{d}_n(t)-d(t)|+|d(t)|\}-d_{\infty,n} \Big\}-\ell_n^2(\Ec_n)\\
	 	&\leq \ell_n(\Ec_n)\frac{\sqrt{nh_n}}{\sigma\|K^*\|_2} \Big\{\sup_{t\in I_n}|\hat{d}_n(t)-d(t)|-\rho_n \Big\}-\ell_n^2(\Ec_n)\\
	 	&=\frac{\ell_n(\Ec_n)}{\ell_n}\bigg(\ell_n\frac{\sqrt{nh_n}}{\sigma\|K^*\|_2}\sup_{t\in I_n}|\hat{d}_n(t)-d(t)| - \ell_n^2\bigg) -\ell_n(\Ec_n)\bigg(\frac{\sqrt{nh_n}}{\sigma\|K^*\|_2}\rho_n-\ell_n+\ell_n(\Ec_n)\bigg)  ,
	 	\end{align*}
	 	which diverges to $-\infty$ by Assumption  \ref{assump:ghat} and  Remark \ref{stationary}(ii), and since $\sqrt{nh_n}\rho_n/\ell_n \to \infty$ by assumption and $\ell_n(\Ec_n)\to\infty$ by Proposition \ref{prop:BlowingUp}(i). Thus,
	 	\begin{equation*}
	 	D_n^\Ec(I_n) = D_n^\Ec(\Ec_n)+o_\pr(1),
	 	\end{equation*}
		and the same assertion is obviously true if  $\Ec_n=I_n$.

	 	 If $\Ec^+\neq \emptyset$, recall the definition of $Z_{n,3}$ in \eqref{hd30} and observe that analogously to \eqref{eq:muHat}, and by  \eqref{eq:zn1zn3} and Assumption~\ref{assump:ghat}, 
	 	\begin{align*}
	 	D_n^\Ec(\Ec^+_n)
	 	&= \ell_n(\Ec_n)\frac{\sqrt{nh_n}}{\sigma\|K^*\|_2} \Big\{\sup_{t\in\Ec^+_n}|\hat{d}_n(t)|- d_{\infty,n} \Big\}-\ell_n^2(\Ec_n)\\
	 	&= \ell_n(\Ec_n)\frac{\sqrt{nh_n}}{\sigma\|K^*\|_2} \Big\{\sup_{t\in\Ec^+_n}\hat{d}_n(t)- d_{\infty,n}   \Big\}-\ell_n^2(\Ec_n)+o_\pr(1)\\
	 	&\leq \ell_n(\Ec_n)\frac{\sqrt{nh_n}}{\sigma\|K^*\|_2}\sup_{t\in\Ec^+_n} \big\{\hat{d}_n(t)-d(t) \big\}-\ell_n^2(\Ec_n) +o_\pr(1)\\
	 	&= \ell_n(\Ec_n)\sup_{t\in\Ec^+_n}Z_{n,3}(t)-\ell_n^2(\Ec_n) +o_\pr(1).
	 	\end{align*}
	 	Analogously, if $\Ec^-\neq\emptyset$, we have
	 	\begin{align*}
	 	D_n^\Ec(\Ec^-_n)
	 	&\leq \ell_n(\Ec_n)\sup_{t\in\Ec^-_n}-Z_{n,3}(t)-\ell_n^2(\Ec_n)  +o_\pr(1).
	 	\end{align*}
	 	Thus, since 
		\[
		(\sup_{t\in\Ec_n^+}Z_{n,3}(t), \sup_{t\in\Ec^-_n}-Z_{n,3}(t)) \stackrel{\Dc}{=} (\sup_{t\in\Ec_n^+}Z_{n,3}(t), \sup_{t\in\Ec^-_n}Z_{n,3}(t))
		\]
		 as $\Ec_n^+$ and $\Ec_n^-$ are disjoint with distance larger than $h_n$, we have
		\begin{align*}
		D_n^\Ec(I_n) 
		&=
		D_n^\Ec(\Ec_n) + o_\Prob(1) = D_n^\Ec(\Ec^-_n) \vee D_n^\Ec(\Ec^+_n) + o_\Prob(1) \\ 
		& \ll_S  
		\ell_n(\Ec_n)\sup_{t\in\Ec_n}Z_{n,3}(t)-\ell_n^2(\Ec_n)  + o_\Prob(1) 
		=
		G_{n,1} + o_\Prob(1).
		\end{align*}
	 	By Proposition \ref{prop:BlowingUp}, $h_n^{-1}\Ec_n$ satisfies the blowing up property. Thus, by Remark \ref{stationary}(i) and (ii), we obtain that $G_{n,1} \weak \Gum_0$ as asserted. The fact that we  may replace $\ll_S$ by $\stackrel{\Dc}{=}$ in the previous display if $\lambda(\Ec)>0$  follows by the same arguments as given in the proof of Theorem \ref{thm:Dn}.
\end{proof}

\begin{proof}[Proof of Corollary \ref{test2}]
We only consider the case $\Delta >0$ (the proof $\Delta =0$ follows by similar arguments). 
Recall that $\hat{\sigma}^2-\sigma^2 = \Oc_\pr(n^{-1/3})$ by \eqref{eq:estLRV}. By Theorem \ref{thm:EnEstimator} and similar calculations as in \eqref{eq:lnEln-ln2} we obtain 
\[
\frac{\ell_n(\hat{\Ec}_n)}{\ell_n(\Ec_n)}\convp 1\quad\text{ and }\quad\ell_n(\hat{\Ec}_n)\ell_n(\Ec_n)-\ell_n^2(\hat\Ec_n) \convp 0.
\] 
Together with Theorem~\ref{conv:DnEn} this yields
\[
\liminf_{n\to \infty}\pr\Big(\tfrac{\ell_n(\hat\Ec_n)\sqrt{nh_n}}{\hat\sigma\|K^*\|_2} \big( \hat d_{\infty, n} - d_{\infty, n}\big)-\ell_n^2(\hat\Ec_n)\leq x\Big) 
\geq 
\begin{cases} 
	\Gum_0((-\infty, x]) &, d_\infty>0, \\
	\Gum_{\log 2}((-\infty, x]) &, d_\infty=0. \\
\end{cases}
\]
Moreover,  we have $|\hat d_{\infty, n} - d_{\infty, n}| = o_\Prob(1)$, which follows from Theorem~\ref{thm:Dn} for $\lambda(\Ec)>0$ and from \eqref{eq:axel4} for $\lambda(\Ec)=0$.
The derived convergences imply the corollary by similar arguments as in the proof of Corollary~\ref{FloZirkus}. 
\end{proof}

\subsection{Proofs for Section \ref{sec5}}

\begin{proof}[Proof of Theorem \ref{thm:estFirstExceedance}]
	Define $\gamma_n=\big(2 \sigma \|K^*\|_2\tfrac{\ell_n }{\sqrt{nh_n}} +\delta_n\big)^{1/\kappa}$ and let  $c_1 = (2/c_\kappa)^{1/\kappa}$, where $\kappa$ and $c_\kappa$ are as in \eqref{eq:smoothnessMu}. 	 First, consider the case  $t^*=\infty$. Then $\Delta- d_{\infty, n} \ge \Delta - d_\infty > 0$ for all $n\in\N$. Hence, since $|\hat{d}_{\infty,n}-d_{\infty,n}|=o_\pr(1)$ (this follows from Theorem~\ref{thm:Dn} for $\lambda(\Ec)>0$ and from \eqref{eq:axel4} for $\lambda(\Ec)=0$), we obtain
	\begin{align*}
	\pr(\hat{t}^*<t^*)
	&= \pr(\hat d_{\infty, n} \geq \Delta -  \delta_n) 
	= \pr(\hat d_{\infty, n} - d_{\infty,n} + \delta_n \geq \Delta - d_{\infty, n} ) =o_\Prob(1).
	\end{align*}

Next, consider the case  $t^*<\infty$. The assertion follows from
	\begin{align}\label{eq:estFirstDev1}
	\pr(\hat{t}^*-t^*<-c_1\gamma_n)&=o(1),\\
	\pr(\hat{t}^*-t^*>h_n)&=o(1).\label{eq:estFirstDev2}
	\end{align}

	For the proof of \eqref{eq:estFirstDev1} note that it follows from \eqref{eq:smoothnessMu} that $|d(t^*)-d(s)|-c_\kappa (t^*-s)^\kappa = o(|t^*-s|^\kappa)$.  A careful case-by-case study 
of the absolute value $|d(t^*)-d(t^*-c_1\gamma_n)|$, depending on whether $d(t^*)=\Delta$ or $d(t^*)=-\Delta$,
then implies
	\[
	|d(t^*-c_1\gamma_n)|=\Delta-c_\kappa (c_1\gamma_n)^\kappa + o(\gamma_n^\kappa) = \Delta - 2 \gamma_n^\kappa+o(\gamma_n^\kappa). 
,
	\] 
	as $\gamma_n\to 0$ by assumption. Thus, by continuity of $d$, compactness of $[h_n\vee x_0,t^*-c_1\gamma_n]$ and Remark \ref{rem:concavity}, it follows for all sufficiently large $n$ that
	\begin{equation*} 
	\max_{s\in[h_n\vee x_0,t^*-c_1\gamma_n]}|d(s)| =
	|d(t^*-c_1\gamma_n)| \le \Delta -  \gamma_n^\kappa.
	\end{equation*} 
	This implies, by the definition of $\hat t^*$, Assumption \ref{assump:ghat} and Remark \ref{stationary} (ii), 
	\begin{align*}
	&\phantom{{}={}} \pr\big(\hat{t}^*<t^*-c_1\gamma_n\big) \\
	&\leq \pr\big(|\hat{d}_n(s)|\geq \Delta-\delta_n~\text{for some}~s \in [h_n\vee x_0, t^*-c_1\gamma_n)\big)\\
	&\leq \pr\big(|d(s)|+|\hat{d}_n(s)-d(s)|\geq \Delta-\delta_n~\text{for some}~s \in [h_n\vee x_0, t^*-c_1\gamma_n)\big)\\
	&\leq \pr\Big(|\hat{d}_n(s)-d(s)|-\gamma_n^\kappa\geq -\delta_n~\text{for some}~s \in [h_n\vee x_0, t^*-c_1\gamma_n)\Big)\\
	&\leq \pr\Big(\tfrac{\ell_n\sqrt{nh_n}}{\sigma \|K^*\|_2} \sup_{s\in I_n}|\hat{d}_n(s)-d(s)|-\ell_n^2 \geq \ell_n\Big(\tfrac{\sqrt{nh_n}( \gamma_n^\kappa-\delta_n)}{\sigma \|K^*\|_2}-\ell_n\Big)\Big)
	\to 0.
	\end{align*}
	For \eqref{eq:estFirstDev2}, observe that similarly,
	\begin{align*}
	&\phantom{{}={}} \pr(\hat{t}^*>t^*+h_n)\\
	&\leq \pr\Big(\max_{t\in[h_n\vee x_0,s]}|\hat{d}_n(t)|<\Delta-\delta_n~\text{for some}~s\geq t^*\Big)\\
	&\leq \pr\Big(\max_{t\in[h_n\vee x_0,s]}|d(t)|-\max_{t\in[h_n\vee x_0,s]}|\hat{d}_n(t)-d(t)|<\Delta-\delta_n~\text{for some}~s\geq t^*\Big)\\
	&\leq \pr\Big(\ell_n\big(\tfrac{\sqrt{nh_n}\delta_n}{\sigma \|K^*\|_2}-\ell_n\big)<\tfrac{\ell_n\sqrt{nh_n}}{\sigma \|K^*\|_2}\sup_{t\in I_n}|\hat{d}_n(t)-d(t)|-\ell_n^2\Big),
	\end{align*}
	which vanishes as the left-hand side in the latter probability diverges to infinity by assumption and the right-hand side is bounded in probability by Remark \ref{stationary}(ii) and Assumption~\ref{assump:ghat}.
\end{proof}

\subsection{Proofs for Section \ref{sec6}}

\begin{proof}[Proof of Theorem \ref{locstat}]
	The convergence rate in the theorem is an improved version of Theorem 4.4 of \cite{dettewu2019} and for the sake of brevity we only state the main idea of the proof. Observe that equation (1.1) in \cite{Fan2003} generalises Burkholder's inequality from quadratic variation to the general case of $p$-variation. Using this generalised version of Burkholder's inequality, the inequalities in Theorem 1 of \cite{Wu2007} are valid not only for $q'=\min(2,q)$, but for $q$ itself. Under Assumption \ref{assump:errorLS}, Theorem~1 of \cite{Wu2007} can now be applied with $q=4$. Thus, the convergence rate in Lemma~3 of \cite{ZhouWu2010} improves to $\Oc\big(m^{1/4}(n\tau_n)^{-1/2}\big)$, which finally allows us to derive the improved convergence of Theorem 4.4 of \cite{dettewu2019} as stated in the theorem.
\end{proof}

\begin{proof}[Proof of Theorem \ref{thm:convLS}]
The proof follows by same arguments  as given in Section \ref{proofsec4}, where one uses Theorem \ref{locstat} instead of \eqref{eq:estLRV} and Theorem \ref{a1} instead of Remark \ref{stationary}. The details are omitted for the sake of brevity.
\end{proof}

\end{appendix}

\section*{Acknowledgements}
This work has been supported in part by the
Collaborative Research Center ``Statistical modeling of nonlinear
dynamic processes'' (SFB 823, Project A1, A7,  C1) of the German Research Foundation (DFG).

\bibliographystyle{chicago}

\bibliography{bibliography}

\end{document}